\documentclass[a4paper,11pt]{article}
\usepackage{amsfonts,amssymb,amsthm,amsmath, dsfont,relsize,accents}
\usepackage{bm}
\usepackage{mathrsfs}  
\usepackage[usenames]{color}
\usepackage{upgreek}
\usepackage[english]{babel}
\usepackage{graphicx}
\usepackage{microtype}
\usepackage[colorlinks=true, pdfstartview=FitV, linkcolor=blue, citecolor=blue, urlcolor=blue,pagebackref=false]{hyperref}
\usepackage{graphicx}
\usepackage{epstopdf}
\usepackage{enumitem}
\usepackage{caption}

\usepackage{tikz,listofitems}
\usetikzlibrary{patterns}
\definecolor{lred}{RGB}{255,186,186}
\definecolor{lblue}{RGB}{186,186,255}
\definecolor{lgrey}{RGB}{220,220,220}

\definecolor{labelkey}{gray}{.8}
\definecolor{refkey}{gray}{.8}

\definecolor{darkred}{rgb}{0.9,0.1,0.1}

 \makeatletter
 \@addtoreset{equation}{section}
 \makeatother

 \makeatletter
 \@addtoreset{enunciato}{section}
 \makeatother

 \newcounter{enunciato}[section]

 \newtheorem{ittheorem}{Theorem}
 \newtheorem{itlemma}{Lemma}
 \newtheorem{itproposition}{Proposition}
 \newtheorem{itcorollary}{Corollary}
 \newtheorem{itdefinition}{Definition}
 \newtheorem{itremark}{Remark}
 \newtheorem{itclaim}{Claim}
 \newtheorem{itfact}{Fact}
 \newtheorem{itconjecture}{Conjecture}

 \newenvironment{theorem}{\addtocounter{enunciato}{1}
 \begin{ittheorem}}{\end{ittheorem}}

 \newenvironment{lemma}{\addtocounter{enunciato}{1}
 \begin{itlemma}}{\end{itlemma}}

 \newenvironment{proposition}{\addtocounter{enunciato}{1}
 \begin{itproposition}}{\end{itproposition}}

 \newenvironment{corollary}{\addtocounter{enunciato}{1}
 \begin{itcorollary}}{\end{itcorollary}}

 \newenvironment{definition}{\addtocounter{enunciato}{1}
 \begin{itdefinition}}{\end{itdefinition}}

 \newenvironment{remark}{\addtocounter{enunciato}{1}
 \begin{itremark}}{\end{itremark}}

 \newenvironment{claim}{\addtocounter{enunciato}{1}
 \begin{itclaim}}{\end{itclaim}}

 \newenvironment{fact}{\addtocounter{enunciato}{1}
 \begin{itfact}}{\end{itfact}}

 \newenvironment{conjecture}{\addtocounter{enunciato}{1}
 \begin{itconjecture}}{\end{itconjecture}}

 \newcommand{\be}[1]{\begin{equation}\label{#1}}
 \newcommand{\ee}{\end{equation}}

 \newcommand{\bl}[1]{\begin{lemma}\label{#1}}
 \newcommand{\el}{\end{lemma}}

 \newcommand{\br}[1]{\begin{remark}\label{#1}}
 \newcommand{\er}{\end{remark}}

 \newcommand{\bt}[1]{\begin{theorem}\label{#1}}
 \newcommand{\et}{\end{theorem}}

 \newcommand{\bd}[1]{\begin{definition}\label{#1}}
 \newcommand{\ed}{\end{definition}}

 \newcommand{\bcl}[1]{\begin{claim}\label{#1}}
 \newcommand{\ecl}{\end{claim}}

 \newcommand{\bfact}[1]{\begin{fact}\label{#1}}
 \newcommand{\efact}{\end{fact}}

 \newcommand{\bp}[1]{\begin{proposition}\label{#1}}
 \newcommand{\ep}{\end{proposition}}

 \newcommand{\bc}[1]{\begin{corollary}\label{#1}}
 \newcommand{\ec}{\end{corollary}}

 \newcommand{\bcj}[1]{\begin{conjecture}\label{#1}}
 \newcommand{\ecj}{\end{conjecture}}

 \newcommand{\bpr}{\begin{proof}}
 \newcommand{\epr}{\end{proof}}

 \newcommand{\bprlem}[1]{\begin{proofof}{\it Lemma \ref{#1}}.\,\,}
 \newcommand{\eprlem}{\end{proofof}}

 \newcommand{\bprthm}[1]{\begin{proofof}{\it Theorem \ref{#1}}.\,\,}
 \newcommand{\eprthm}{\end{proofof}}

 \newcommand{\bprprop}[1]{\begin{proofof}{\it Proposition \ref{#1}}.\,\,}
 \newcommand{\eprprop}{\end{proofof}}

 \newcommand{\bi}{\begin{itemize}}
 \newcommand{\ei}{\end{itemize}}

 \newcommand{\ben}{\begin{enumerate}}
 \newcommand{\een}{\end{enumerate}}

 \newenvironment{proofof}{\noindent {\em Proof of\,\,}}{\hspace*{\fill}$\halmos$\medskip}
 \newcommand{\halmos}{\rule{1ex}{1.4ex}}

 \newcommand{\one}{{\mathchoice {1\mskip-4mu\mathrm l}
         {1\mskip-4mu\mathrm l}
         {1\mskip-4.5mu\mathrm l}
         {1\mskip-5mu\mathrm l}}}

\def \E {{\mathbb E}}

\def \N {{\mathbb N}}
\def \P {{\mathbb P}}

\def \R {{\mathbb R}}
\def \Z {{\mathbb Z}}

\def \lra \leftrightarrow

\def \ra {\rightarrow}
\def \ba {\begin{array}}
\def \ea {\end{array}}

\def \lra {\longrightarrow}

\def \var {{\rm var}}

\def \cA {{\mathcal A}}
\def \cC {{\mathcal C}}

\def \cG {{\mathcal G}}
\def \cH {{\mathcal H}}

\def \cO {{\mathcal O}}
\def \cR {{\mathcal R}}

\def \lra {{\leftrightarrow}}

\def \H {{\mathcal{H}}}
\def \B {{\mathcal{B}}}

\def \subset {\subseteq}

\def \emptyset {\varnothing}

\def\one{\rlap{\mbox{\small\rm 1}}\kern.15em 1}

\newlength{\dhatheight}

\newcommand{\NN}{\mathbb{N}}
\newcommand{\RR}{\mathbb{R}}
\newcommand{\EE}{\mathbb{E}}
\newcommand{\ZZ}{\mathbb{Z}}
\newcommand{\dist}{\mathrm{d_h}}

\newcommand{\poimod}{\mathrm{Poi}}

\begin{document}

\title{The contact process on random hyperbolic graphs: metastability and critical exponents}

\author{Amitai Linker\textsuperscript{1}, Dieter Mitsche\textsuperscript{2}, Bruno Schapira\textsuperscript{3}, Daniel Valesin\textsuperscript{4}}
\footnotetext[1]{\noindent Institut Camille Jordan, Universit\'e Jean Monnet, Univ. de Lyon, France.\\ \url{amitailinker@gmail.com}. Research partially supported by IDEXLYON of Universit\'{e} de Lyon (Programme Investissements d'Avenir ANR16-IDEX-0005).}
\footnotetext[2]{\noindent Institut Camille Jordan, Universit\'e Jean Monnet, Univ. de Lyon, France.\\ \url{dmitsche@unice.fr}. Research partially supported by IDEXLYON of Universit\'{e} de Lyon (Programme Investissements d'Avenir ANR16-IDEX-0005) and by Labex MILYON/ANR-10-LABX-0070.}
\footnotetext[3]{Aix-Marseille Universit\'e, CNRS, Centrale Marseille, I2M, UMR 7373, 13453 Marseille, France.\\ \url{bruno.schapira@univ-amu.fr}}
\footnotetext[4]{\noindent University of Groningen, Nijenborgh 9, 9747 AG Groningen, The Netherlands.\\ \url{d.rodrigues.valesin@rug.nl}}

\maketitle
\begin{abstract}
We consider the contact process on the model of hyperbolic random graph, in the regime when the degree distribution obeys a power law with exponent $\chi \in(1,2)$ (so that the degree distribution has finite mean and infinite second moment). 
We show that the probability of non-extinction as the rate of infection goes to zero decays as a power law with an 
exponent that only depends on $\chi$ and which is the same as in the configuration model, suggesting some universality of this critical exponent.  
We also consider finite versions of the hyperbolic graph 
and prove metastability results, as the size of the graph goes to infinity.  
\end{abstract}

\section{Introduction}\label{sec:intro}

It has been empirically observed that complex networks such as social networks, scientific collaborator networks, citation networks, computer networks and others (see~\cite{Z2}) typically are scale-free and exhibit a non-vanishing clustering coefficient. Moreover, these networks have a heterogeneous degree structure, the typical distance between two vertices is very small, and the maximal distance is also small. A model of complex networks that naturally exhibits these properties is the random hyperbolic model introduced by~\cite{KPKVB} (and later formalized by~\cite{GPP}): 
one convincing demonstration of this fact was given by Bogu\~n\'{a}, Papadopoulos, and Krioukov in~\cite{BPK}  where a compelling maximum likelihood fit of autonomous systems of the internet graph in hyperbolic 
  space was computed. 
Another important aspect of this random graph model is its mathematically elegant specification, making it amenable to mathematical analysis.
This partly explains why the model has been studied also analytically by theoreticians.

On the other hand, the contact process describes a class of interacting particle systems which serve as a model for the spread 
of epidemics on a graph. Its use in the context of complex networks as above goes back at least to Berger, Borgs, Chayes and Saberi \cite{BBCS05}, and has been since then the object of an intense activity (see below for a partial overview).

Before giving more related work, we define the concepts mentioned in more detail.

\subsection*{The hyperbolic graph model of~\cite{KPKVB}}
In the original model of Krioukov, Papadopoulos, Kitsak, Vahdat, and Bogu\~{n}\'{a}~\cite{KPKVB} an $n$-vertex size graph was obtained by first randomly choosing $n$ points in the disk of radius $R=R(n)$ centered at the origin of the hyperbolic plane.
From a probabilistic point of view it is arguably more natural to consider the Poissonized version of this model. Formally, the Poissonized model 
is the following (see also~\cite{GPP} for the same description in the uniform model): for each $n \in \NN$, consider a Poisson point process on the hyperbolic disk of radius $R :=2 \log (n/\nu)$ for some positive constant $\nu \in \RR^+$ ($\log$ denotes here and throughout the paper the natural logarithm) and denote its point set by $V_n$ (the choice of $V_n$ is due to the fact that we will identify points of the Poisson process with vertices of the graph).

The intensity function at polar coordinates $(r,\theta)$ for 
  $0\leq r< R$ and $0 \leq \theta < 2\pi$ is equal to
\[
g(r,\theta) := \nu e^{\frac{R}{2}}f(r,\theta),
\]
where $f(r,\theta)$ is the joint density function with $\theta$ chosen uniformly at random in the interval $[0,2\pi)$ and independently of $r$, which is chosen according to the density function
\begin{align*}
f(r) & := \begin{cases}\displaystyle
   \frac{\alpha\sinh(\alpha r)}{\cosh(\alpha R)-1}, &\text{if $0\leq r< R$}, \\
   0, & \text{otherwise}.
  \end{cases}
\end{align*}
Note that this choice of $f(r)$ corresponds to the uniform distribution inside a disk of radius $R$ around the origin in a hyperbolic plane of curvature $-\alpha^2$. Identify then the points of the Poisson process with vertices
(that is, identify a point with polar coordinates $(r_v,\theta_v)$ with vertex $v\in V_n$) and make the following graph $G_n=(V_n,E_n)$: for $u, u'\in V_n$, $u \neq u'$, there is an edge with endpoints 
  $u$ and $u'$ provided the distance (in the hyperbolic plane) between
  $u$ and $u'$ is at most $R$, i.e.,  
  the hyperbolic distance
  between $u$ and $u'$, denoted by 
  $\dist:=\dist(u,u')$,
  is such that  $\dist\leq R$ where $\dist$ is obtained by solving 
\begin{equation}\label{eqn:coshLaw}
\cosh \dist := \cosh r_u\cosh r_{u'}-
  \sinh r_u\sinh r_{u'}\cos( \theta_u{-}\theta_{u'}).
\end{equation}

For a given $n \in \NN$, we denote this model by 
  $\poimod_{\alpha,\nu}(n)$.
Note in particular that 
\[
\iint g(r,\theta) \, d\theta\, dr 
  = \nu e^{\frac{R}{2}}=n,
\]
and thus  $\EE|{V_n}|=n.$ 
  The main advantage of defining $V_n$ as a Poisson point process is
  motivated by the following two properties: the number of points of
  $V_n$ that lie in any region $A$ follows a Poisson
  distribution with mean given by $\int_A g(r,\theta) \, dr\, d\theta$, and the numbers of points of $V_n$ in disjoint
  regions of the hyperbolic plane are independently distributed.
  
  In this paper we restrict ourselves to $\frac12 < \alpha < 1$. The restriction $\alpha>\frac12$ guarantees that the resulting graph has bounded average degree (depending
  on $\alpha$ and $\nu$ only): if $\alpha<\frac12$, then the degree sequence is so 
  heavy tailed that this is impossible (the graph is with high probability connected in this case, as shown in~\cite{BFM16}), and if $\alpha>1$, then
  as the number of vertices grows,
  the largest component of a random hyperbolic graph has sublinear size
(more precisely, its order is $n^{1/(2\alpha)+o(1)}$, see~\cite[Theorem~1.4]{BFM15} and~\cite{Diel}). 
  It is known that for $\frac12 < \alpha < 1$, with high probability the 
  graph $G_n$ has a linear size 
  component~\cite[Theorem~1.4]{BFM15}  
  and the second largest component has size
 $ \Theta(\log^{\frac{1}{1-\alpha}} n)$~\cite{KM19+},
  which justifies referring to the
  linear size component as \emph{the giant component}. More
  precise results including a law of large numbers for the largest
  component in these networks were established  in~\cite{FM}.

For ease of notation, we will assume $\nu=1$ throughout the paper;  all our results, however, hold for any constant $\nu$.  
In fact, in this paper, we use a different representation, namely the representation of the hyperbolic graph in the upper half-plane. 
For our purposes, the representations are equivalent (see Section~\ref{sec:prelim} for details), and for us it is easier to deal with the latter. 
We consider an infinite rooted version of this graph (that is, a graph in which one vertex is distinguished as the root, once more see Section~\ref{sec:prelim} for details), which we shall denote by ${\bf G}_\infty$, and a finite version, corresponding to the previous model: for $n\ge 0$, we let ${\bf G}_n$ denote the restriction of ${\bf G}_\infty$ to the rectangle $[-\frac{\pi}{2}n,\frac{\pi}2 n]\times [0,2\log n]$, in which we identify the left and right boundaries.

 \subsection*{The contact process}

In the contact process, each vertex of a graph is at any point in time either healthy (state~0) or infected (state~1). 
The continuous-time dynamics is defined by the specification that infected vertices become healthy with rate one, and transmit the infection to 
each neighboring vertex with rate $\lambda > 0$. We refer to \cite{Lig} for a standard reference on the contact process.

Given a subset $A$ of the set of vertices $V$ of a graph, we denote by $(\xi_t^A)_{t\ge 0}$ the contact process  
starting from an initial configuration of infected vertices equal to $A$, and write simply $(\xi_t^v)_{t\ge 0}$ when $A$ is a singleton $\{v\}$ (when a superscript is not present, the initial configuration is either clear from the context or unimportant). We will view $\xi_t^A$ either as a function from $V$ to $\{0,1\}$, or as a subset of $V$.

\subsection*{Our results}

Our first result concerns the non-extinction probability of the contact process on ${\bf G}_\infty$, starting from only the root infected, which we denote by $\gamma(\lambda)$. In particular, it shows that $\gamma(\lambda)$ is nonzero for all $\lambda>0$, which means that the critical infection rate $\lambda_c({\bf G}_\infty)$ is almost surely equal to $0$. Thus Theorem~\ref{thm:gammafraction} should be read as a result on the asymptotic behavior of $\gamma(\lambda)$, as $\lambda$ approaches this critical value by above. Given non-negative functions~$\lambda\mapsto f(\lambda),g(\lambda)$, we say that~$f(\lambda) \asymp g(\lambda)$ as~$\lambda \to 0$ if there exist two positive constants~$c$ and~$C$ such that~$cf(\lambda) \le g(\lambda) \le Cf(\lambda)$ for all~$\lambda$ small enough.
\begin{theorem}\label{thm:gammafraction}
As~$\lambda \to 0$,
$$\gamma(\lambda)  \asymp \begin{cases}
\lambda^{\frac{1}{2-2\alpha}},&\alpha \in (\tfrac12,\tfrac34];\\[.2cm]
\frac{\lambda^{4\alpha - 1}}{\log(1/\lambda)^{2\alpha - 1}}& \alpha \in (\tfrac34,1). \end{cases}$$
\end{theorem}
It is worth noting that such result has been shown in only a very limited number of other examples. Indeed, to our knowledge so far it was only 
established for the configuration model \cite{CanS, CD, mvy13}, and the 
so-called P\'{o}lya point graph~\cite{Can} (which is the local limit of preferential attachment graphs~\cite{BBCS14}), 
as well as for certain classes of dynamical networks \cite{JLM19}. We shall comment further on the similarities and differences between all these results a bit later; 
in particular the exponent in the power of $\lambda$ seems to be a universal constant only depending on the degree distribution, 
while the power of the logarithmic correction seems on the contrary to be model dependent.

Our next results concern finite versions of the hyperbolic random graph and show metastability type results, namely that the extinction time when starting from the fully occupied configuration is exponential in the size of the graph (see Theorem~\ref{theo.exp}), 
and furthermore that the density of infected sites remains close to $\gamma(\lambda)$ for an exponentially long time (see Theorem~\ref{thm:convergence}).
 
For a finite graph $G$, we define $\uptau_G$ as the extinction time of the contact process on $G$, when starting from all vertices infected. This is the hitting time of the unique absorbing state of the process, equal to the identically zero configuration.
\begin{theorem}\label{theo.exp}
For any~$\lambda > 0$ and~$\alpha \in (\tfrac12,1)$, there exist~$c > 0$ and $\beta\in(0,1)$, such that
$$\P(\uptau_{{\bf G}_n}>e^{cn}) > 1-e^{-cn^{\beta}}, \quad \forall n\ge 1.$$
\end{theorem}

The next result shows that there is no hope to take $\beta=1$ in Theorem \ref{theo.exp}. 

\begin{proposition}
	\label{badgn}
	For any $\alpha\in (1/2,1)$, there are $\beta,\varepsilon'\in(0,1)$ and a ${\bf G}_n$-measurable event $A_n$ with probability $\P(A_n)>\exp(-n^\beta)$, 
	such that  $\E[\tau_{{\bf G}_n}\mid A_n]<\exp(n^{\varepsilon'})$. 
\end{proposition}

Finally our last main result proves the convergence of the density of infected sites to the non-extinction probability on the infinite graph ${\bf G}_\infty$. 

\begin{theorem}\label{thm:convergence}
For any~$\lambda > 0$ and~$\alpha \in (\tfrac12,1)$, there exists~$c > 0$ such that the following holds. Fix~$(t_n)_{n \geq 1}$ such that~$t_n \to \infty$ and~$t_n < e^{cn}$ for each~$n$. Then, for any~$\varepsilon > 0$,
$$\P\left( \Big|\frac{|\xi^{\bf G_n}_{t_n}|}{n} - \gamma(\lambda) \Big| > \varepsilon \right)   \ \underset{n \to \infty}{\longrightarrow} \ 0.$$
\end{theorem}
Metastability results such as Theorems~\ref{theo.exp} and~\ref{thm:convergence} for the contact process  were first established in $1984$ for finite intervals of the line~\cite{CGOV84}, and have since then been obtained in a large number of other examples, 
including finite boxes of $\Z^d$ (see~\cite{DS88, mo93} and references therein), finite regular trees~\cite{CMMV13,St01}, 
random regular graphs~\cite{LS15,mv14}, the configuration model~\cite{CanS,CD,mvy13}, Erd\'os-Renyi random graphs \cite{BNNS}, 
preferential attachment graphs~\cite{Can}, rank-one inhomogeneous random graphs \cite{Can2}, as well as for a large class of general finite graphs \cite{MMVY, SV}. The general idea of the proof is often similar in all these models, 
but the technical difficulties are specific to each case. Here as well, the hyperbolic nature of the graphs we consider lead to some new 
difficulties.

\subsection*{Overview of proofs}
The proof of Theorem~\ref{thm:gammafraction} is based on proving corresponding lower and upper bounds. For the lower bounds, 
we use a standard argument: we show that there is a certain chance that the root will infect a vertex of sufficiently large degree, from 
where on the infection then survives; either directly infecting from there vertices of even higher degree, or indirectly infecting such vertices 
using low degree vertices, therefore giving rise to two different regimes. The upper bounds require some harder and more original work.
They are based first on partitioning the event of survival into different events, depending essentially on the distance to the origin and the degree of the vertices which are  
reached by the contact process, in such a way that each of the events has at most the desired probability 
to happen. Again, in both regimes we identify different events, giving rise to different values. 
Also, interestingly our estimates rely on some new facts 
about the non-extinction probability of the contact process which hold on general graphs 
and which might as such be of independent interest; see in particular Lemma~\ref{lem:star_lem}. 

The proof of Theorem~\ref{theo.exp} is based on finding a large (linear-sized)  connected subgraph on which the contact process survives for a long time. The key idea is a suitable tessellation of the upper half-plane into different boxes, such that a constant proportion of small degree vertices belongs to this subgraph, and such that all vertices of sufficiently large degree belong to this graph as well. Proposition~\ref{badgn} is shown by explicitly constructing a graph whose connected components are of size at most $cn^{1-\alpha}$, therefore yielding a smaller extinction time.

Finally, Theorem~\ref{thm:convergence} makes use of the idea that if the process on the infinite graph starting from only the root infected, survives for a long time then and only then it will escape from a large neighborhood of the root. The proof of this idea is based on self-duality of the contact process, and then by applying the first and second moment methods to the number of vertices escaping from a large neighborhood (for corresponding upper and lower bounds, respectively); The hyperbolic shapes of the neighborhoods, however, and in particular, the existence of very high-degree vertices make this basic idea a bit delicate at times.

\subsection*{Discussion of results}
In Theorem~\ref{thm:gammafraction} we can observe a phase transition at $\alpha=\frac34$. This is interesting for different reasons: recently it was observed that  the value of $\alpha=\frac34$ corresponds to a change of regime in the local clustering coefficient averaged over all vertices of degree exactly $k$ (see~\cite{Schepers} for details) - for $\alpha > \frac34$ the clustering coefficient is of the order $\frac{1}{k}$, whereas for $\frac12 < \alpha < \frac34$ it is of the order $k^{2-4\alpha}$ (for $\alpha=\frac34$ it is of the order $\log k/k$). It would be interesting to investigate further the link between these two results. Second, since random hyperbolic graphs have a power law degree distribution with exponent $\chi:=2\alpha+1$ (see~\cite{GPP}), the phase transition given here as well as the speed of decay to zero of $\gamma(\lambda)$ is exactly the same as in the configuration model 
\cite{mvy13}, for both regimes. Given the similarities in the proof strategies in the two models this might perhaps be less surprising, but it clearly raises the natural question whether a more general theorem, with more general conditions on a random graph model, 
can be stated and proved. In fact this striking fact had already been observed in another model, the P\'olya-point graph, already 
mentioned before. Indeed, in~\cite{Can} 
it is shown that for $\chi\in [3,+\infty)$, the non-extinction probability also decays polynomially as a function of $\lambda$, with the same exponent 
as in the configuration model~\cite{mvy13}, except for the power of the logarithmic correction, 
which suggests that only the power of $\lambda$ might be a 
universal constant. 

\textbf{Related work.}
Although the random hyperbolic graph model was relatively
  recently introduced~\cite{KPKVB}, several of its key properties have already been established. As already mentioned, in~\cite{GPP}, the degree distribution, the expected value of the maximum degree and global 
  clustering coefficient were determined (details on the local clustering coefficient were then established recently in the already mentioned paper of~\cite{Schepers}), and in~\cite{BFM15}, the existence of a giant component as a function of $\alpha$.
  
  The threshold in terms of $\alpha$ for the connectivity of random
  hyperbolic graphs was given in~\cite{BFM16}. 
The logarithmic diameter of the giant component was established in~\cite{MSt}, whereas  the average distance of two points belonging to the giant component
  was investigated in~\cite{ABF}. 
Results on the global clustering coefficient of the so called
  binomial model of random hyperbolic graphs were obtained
  in~\cite{CF16}, and on the evolution of graphs on more general
  spaces with negative curvature in~\cite{F12}. 
Finally, the spectral gap of the Laplacian of this model was studied 
  in~\cite{KM18}.

The model of random hyperbolic graphs for $\frac12 < \alpha < 1$ is very similar to two different models
studied in the literature: the model of inhomogeneous long-range
percolation in $\ZZ^d$ as defined in~\cite{Remco}, and the
model of geometric inhomogeneous random graphs, as introduced
in~\cite{BKL19} (see these papers and the references therein for more details about these models). In both cases, each vertex is given a weight, and
conditionally on the weights, the edges are independent (the presence
of edges depending on one or more parameters). The latter model generalizes random hyperbolic graphs.

\subsection*{Plan of the paper}
The paper is organized as follows. In Section~\ref{sec:prelim}, we define more precisely the random graph models on which we will work. We also recall  basic facts and definitions about them, as well as for the contact process. In Section~\ref{Section.exp}, we prove Theorem~\ref{theo.exp} and Proposition~\ref{badgn}, which are based on some basic geometric constructions that shall be used throughout the paper. In Sections~\ref{sec.lower} and~\ref{sec.upper}, we prove the lower and upper bounds in Theorem~\ref{thm:gammafraction}, respectively.
Finally Section~\ref{sec.convergence} provides the proof of Theorem~\ref{thm:convergence}.

\section{Preliminaries}\label{sec:prelim}

\subsection{Hyperbolic graph model}
Following~\cite{FM}, we consider the continuum percolation model defined in the upper half-plane. Thus we let 
$$\mathbb H:=\R\times [0,\infty),$$ and consider an inhomogeneous Poisson Point Process $\mathcal{P}$ on $\mathbb H$ with intensity measure $\mu$ given by 
$$d\mu(x,h) = \frac {\alpha}{\pi}e^{-\alpha h} \, dx \, dh.$$
The first coordinate of a point in $\mathbb H$ is sometimes called its {\it horizontal coordinate} (or {\it $x$-coordinate}), 
and the second one its {\it height}. 
We then define ${\bf G}_\infty$ be the graph whose vertex set is the set of points of $\mathcal{P}$, together with an additional (random) point $\rho=(0,{\bf h})$, called the root, 
where ${\bf h}$ is a random variable with density with respect to Lebesgue measure given by $\alpha e^{-\alpha h}$. Furthermore, two vertices $v=(x,h)$ and $v'=(x',h')$ are connected by an edge in ${\bf G}_\infty$ if, and only if,
$$|x-x'| \le e^{(h+h')/2}.$$
For $n\in \N$, we define the graph ${\bf G}_n$, as the restriction of ${\bf G}_\infty$ to the rectangle $[-\frac{\pi}{2}n,\frac{\pi}{2}n]\times[0,2\log n]$, in
which we identify the left and right boundaries. Note that this may create new edges between pairs of vertices which are close to the boundaries.

In~\cite{FM} a precise correspondance is established between ${\bf G}_n$ and the model discussed in the introduction, which indicates that all results that we prove here for ${\bf G}_n$ hold as well for the former model.   

Recall that we set $\nu=1$ and thus $R=2\log n$.
Consider the map $\Psi: [0, R] \times (-\pi, \pi] \to (-\frac{\pi}{2}n, \frac{\pi}{2}n] \times [0, R]$, with 
$$
\Psi: (r, \theta) \mapsto (\theta \frac{e^{R/2}}{2}, R-r), 
$$
between the Poissonized hyperbolic graph model from the introduction and the continuum percolation model in the upper half-plane. 
Denote by $\bf V_n$ the vertex set  of $\bf G_n$. In~\cite{FM} the following result is shown:

\begin{proposition}[\cite{FM}]
There exists a coupling of $G_n$ and ${\bf G}_n$, such that with probability tending to $1$, as $n\to \infty$, 
\begin{itemize} 
\item  $\Psi(V_n)={\bf V}_n$, and  
\item under the event above, for all $u=(r,\theta)$ and $v=(r',\theta')\in V_n$, with $r,r'\ge 3R/4$, $u$ and $v$ are neighbors in $G_n$, if and only if $\Psi(u)$ and $\Psi(v)$ are neighbors in ${\bf G}_n$.    
\end{itemize}
\end{proposition}
Since the proof of Theorem \ref{thm:convergence} only involves vertices at height smaller than $\varepsilon \log n$ with $\varepsilon$ some small constant, the proposition above is enough to transfer our proofs from ${\bf G}_n$ to $G_n$. Theorem~\ref{theo.exp} and Proposition~\ref{badgn} require explicit control of the probabilities of certain bad events. The coupling is not enough to directly transfer the results; however, the proofs of both results can be easily modified for $G_n$, so for consistency we give the proofs still in ${\bf G}_n$.

Now for a vertex $v=(x,h) \in {\bf G_\infty}$, we denote by $\B_\infty(v,1)$ the ball centered at $v$ containing its neighbors, that is,
$$
\B_\infty(v,1):= \{v'=(x',h') \in {\bf G}_\infty : |x-x'| \le e^{(h+h')/2} \}.
$$
More generally, for $r \in \NN$, we let $\B_\infty(v,r)$ denote the subset of vertices of ${\bf G}_\infty$ being at graph distance $r$ from $v$, that is, the set of vertices that can be reached from~$v$ by a path of length at most $r$. 
As in the infinite case, we define for any $r>0$, and any vertex $v \in {\bf G}_n$, by $\B_n(v,r)$ for the ball of graph distance $r$ in ${\bf{G}}_n$.

We need one more fact. Define a rooted graph as a couple $(G,\rho)$, with $G$ some graph and $\rho$ some (possibly random) 
distinguished  vertex of $G$. 
A finite rooted graph $(G,\rho)$ is said to be uniformly rooted, if $\rho$ is a vertex chosen uniformly at random among the vertices of $G$. 
A sequence of rooted graphs $(G_n,\rho_n)_{n\ge 1}$ is said to converge locally towards $(G_\infty,\rho)$ if for every fixed $r > 0$ 
and every fixed graph $H$, $\lim_{n \to \infty} \mathbb{P}(\B_n(\rho,r) \cong H) = \mathbb{P}(\B_\infty(\rho,r) \cong H)$.
In our case it readily follows from the definitions of ${\bf G}_n$ and ${\bf G}_\infty$, that the following holds. 
\begin{lemma}\label{BSconvergence}
The rooted graph $({\bf G}_\infty, \rho)$ is the local limit of the sequence of uniformly rooted graphs~$({\bf G}_n,\rho_n)_{n\ge 1}$, as $n\to \infty$.
\end{lemma}

\subsection{Contact process} \label{ss:contact}
Here we recall some elementary facts about the contact  process, as well as some results  from~\cite{mvy13}. We will keep using the abuse of notation that identifies, for a set~$S$, the element~$\xi \in\{0,1\}^S$ with the set~$\{x \in S:\xi(x)=1\}$.

 Given a graph~$G = (V,E)$ and~$\lambda > 0$, a graphical construction for the contact process on~$G$ with rate~$\lambda$ is a family of Poisson point processes on~$[0,\infty)$:
\begin{align*}
&D^x:\;x \in V \text{ all with rate one, and}\\
&D^{(x,y)}:\;x,y \in V,\;\{x,y\} \in E \text{ all with rate }\lambda;
\end{align*}
all these processes are independent. If~$t \in D^{x}$ we say that there is a recovery mark at~$x$ at  time~$t$ (or in short, at~$(x,t)$), and if~$t \in D^{(x,y)}$ we say that there is a transmission arrow from~$x$ to~$y$ at time~$t$ (or in short, from~$(x,t)$ to~$(y,t)$). An infection path in the graphical construction is a right-continuous, constant-by-parts function~$g:I\to V$ for some interval~$I$, so that: 
\begin{align*}
&-\text{ for all }r \in I,\text{ there is no recovery mark at }(g(r),r);\\
&- \text{ whenever }\gamma(r) \neq \gamma(r^-),\text{ there is a transmission arrow }\\
&\hspace{5cm} \text{ from } (\gamma(r^-),r) \text{ to }(\gamma(r),r).
\end{align*}
Given~$(x,s),(y,t)\in V\times [0,\infty)$ with~$0\le s \le t$, we write~$(x,s) \rightsquigarrow (y,t)$ either if~$(x,s)=(y,t)$ or in the event that there is an infection path~$g:[s,t]\to V$ with~$g(s) = x$ and~$g(t) = y$. For~$A \subset V$, we write~$A \times \{s\} \rightsquigarrow (y,t)$ if we have~$(x,s) \rightsquigarrow (y,t)$ for some~$x \in A$. Similarly we write~$(x,s) \rightsquigarrow B \times \{t\}$ and~$A\times \{s\} \rightsquigarrow B \times \{t\}$.

Given any initial configuration~$A \subset V$, the contact process started from~$A$ infected can be defined from the graphical construction by setting
$$\xi_t^A(x) = {\bf 1}\{A \times \{0\} \rightsquigarrow (x,t)\},\quad t \ge 0,\;x\in V;$$
as mentioned earlier, we write~$\xi^{x}_t$ when~$A  = \{x\}$, and we omit the superscript when it is clear from the context or unimportant.

Due to the invariance of Poisson point processes under time reversal, for any~$A,B \subset V$ we have~$\mathbb{P}(A \times \{0\} \rightsquigarrow B \times \{t\}) = \mathbb{P}(B \times \{0\} \rightsquigarrow A \times \{t\})$; this immediately gives the \textit{self-duality} relation~$\mathbb{P}(\xi^A_t \cap B \neq \varnothing) = \mathbb{P}(\xi^B_t \cap A \neq \varnothing)$. In case~$B = \{x\}$, this gives
\begin{equation}\label{eq:duality_formula}
\mathbb{P}\left(\xi^A_t(x) = 1\right) = \mathbb{P}\left(\xi^x_t \cap A \neq \varnothing\right).
\end{equation}

Let us also repeat the definition of the extinction time
\begin{equation*}
\tau_G := \inf\{t: \xi^V_t = \varnothing\},
\end{equation*}
that is, the time it takes for the process started from all infected to reach the (absorbing) all-healthy configuration.

We now state a result about the contact process on star graphs.
\begin{lemma}\label{lem:first_star}
There exists~$\bar{c} > 0$ such that the following holds for any~$\lambda < 1$ and any~$d \ge 1/(\bar{c}\lambda^2)$. Let~$S_d$ denote the star graph consisting of a center vertex~$o$ with~$d$ neighbors, and let~$(\xi_t)_{t \ge 0}$ denote the contact process with rate~$\lambda$ on~$S_d$. Then, 
\begin{equation}\label{eq:from_mvy_1}
|\xi_0| > \bar{c}\lambda d\quad \Longrightarrow \quad \mathbb{P}\left(|\xi_t| > \bar{c}\lambda d\right) > 1 - \exp\{-\bar{c}\lambda^2d\} \text{ for any }t \in [1,\exp\{\bar{c}\lambda^2 d\}].
\end{equation}
Moreover,
\begin{equation}\label{eq:from_mvy_2}
\xi_0 = {\{o\}}\quad \Longrightarrow \quad\mathbb{P}\left(|\xi_t| > \bar{c}\lambda d\right) > \frac14  \text{ for any }t \in [1,\exp\{\bar{c}\lambda^2 d\}].
\end{equation}
\end{lemma}
Since the proof is essentially the same as that of Lemma~3.1 in~\cite{mvy13}, we omit it. 

We will need the following consequence of the above lemma. For $d \ge 1$, denote by $\mathbb{L}_d$ the graph formed by the half line~$\N_0 = \{0,1,\ldots\}$, where to each vertex $m\in \N_0$, we attach~$d$ additional neighbors (with the additional neighbors attached to distinct points of~$\N_0$ being all distinct).
\begin{lemma}\label{contactNstar} There exist positive constants $c$ and $C$ such that for any $\lambda < 1/2$, the contact process with infection rate $\lambda$ 
survives with probability at least $c$ on the graph $\mathbb{L}_d$, when starting from the origin infected, where $d=C\log(1/\lambda)\cdot \lambda^{-2}$. 
\end{lemma}
\begin{proof}
Let~$C > 0$ be large, to be fixed later. Fix~$\lambda < 1/2$, define~$d$ as in the statement of the lemma and let~$(\xi_t)_{t \ge 0}$ denote the contact process on~$\mathbb{L}_d$ with~$\xi_0 = \{0\}$. 

Define~$t_n := 1+\exp\{\bar{c}\lambda^2 d\}\cdot n$ for all~$n \in \mathbb{N}_0$, where~$\bar{c}$ is the constant of Lemma~\ref{lem:first_star}, and define the discrete-time process
$$\zeta_n(m) := {\bf 1}\{|\xi_{t_n} \cap S_m| \ge \bar{c}\lambda d\},\quad n \in \mathbb{N}_0,\;m \in \mathbb{N}_0,$$
where~$S_m$ denotes the subgraph of~$\mathbb{L}_d$ consisting of the star graph containing~$m \in \N_0$ and its~$d$ extra neighbors (so not including the neighbors of~$m$ in~$\N_0$).  Note that~\eqref{eq:from_mvy_2} gives~$\mathbb{P}\left(\zeta_0(0) = 1\right) = \mathbb{P}\left(|\xi_1 \cap S_0| > \bar{c}\lambda d\right) > \frac14.$

Now, assume that  for some~$m,n$ we have~$\zeta_n(m) = 1$, that is,~$|\xi_{t_n} \cap S_m| \ge \bar{c}\lambda d$. Then, by~\eqref{eq:from_mvy_1}, with probability larger than~$1 - \exp\{-\bar{c}\lambda^2 d\} = 1 - \lambda^{\bar{c}C}$ we also have~$\zeta_{n+1}(m) = 1$.

Moreover, in case~$\zeta_n(m)=1$ and~$\zeta_n(m+1) = 0$, there is a high probability that the infection from~$S_m$ at time~$t_n$ will pass to~$S_{m+1}$ in the time interval~$[t_n, t_{n+1}]$ and occupy it sufficiently long to produce~$\zeta_{n+1}(m+1) = 1$. Indeed, as already mentioned,  the infection remains in~$S_m$ during~$[t_n,t_{n+1}]$ with probability larger than~$1-\lambda^{\bar{c}C}$; condition on this. During this time interval, we make propagation trials as follows: starting a trial at a time~$t \in [t_n,t_{n+1}-3]$, we demand that during~$[t,t+1]$ some infected vertex of~$S_m$ infects~$m$; next, before time~$t+2$ and before recovering,~$m$ infects~$m+1$; finally, the infection spreads in~$S_{m+1}$ until time~$t+3$, so that~$|\xi_{t+3}\cap S_{m+1}| > \bar{c}\lambda d$. The probability of success of such a trial is larger than~$c\lambda^2$ for some~$c > 0$, by~\eqref{eq:from_mvy_2}. The number of trials available is~$\lfloor (t_{n+1}-t_n)/3 \rfloor = \lfloor \exp\{\bar{c}\lambda^2 d\}/3 \rfloor = \lfloor (1/\lambda)^{\bar{c}C}/3 \rfloor$.  Hence, by taking~$C$ large enough and recalling that~$\lambda < 1/2$, the probability to have a successful trial can be made as close to one as desired.

Using these considerations, the proof is completed with a standard argument, showing that~$(\zeta_n)_{n \in \N_0}$ stochastically dominates a  site percolation process~$(\tilde \zeta_n)_{n \in \N_0}$ on the oriented graph with vertex set~$\N_0 \times \N_0$ and all oriented edges of the form~$\langle (m,n),(m,n+1)\rangle$ and~$\langle (m,n),(m+1,n+1) \rangle$. This process can be taken one-dependent, and so that the probability of any site being open is above~$1-\delta$, for any fixed~$\delta > 0$, by taking~$C$ large enough (and uniformly over~$\lambda \in (0,1/2)$). Consequently, it has an infinite percolation cluster containing the origin with positive probability if~$\delta$ is small enough (see~\cite[pages 13-16]{Lig}).
\end{proof}


\section{Proofs of Theorem~\ref{theo.exp} and Proposition~\ref{badgn}}\label{Section.exp}

Our approach for proving Theorem~\ref{theo.exp} consists in showing that there are some $c>0$ and $\beta\in(0,1)$ such that with probability at least $1-e^{-cn^{\beta}}$ the random graph ${\bf G}_n$ is ``good" in the sense that it contains a special structure where the process is able to survive for an exponentially long time in $n$.

In order to find such a structure fix $0<\varepsilon<\frac 1{\log 2}$, and set $L:=\frac{\alpha+1}{2\alpha}\cdot \log 2$, which is chosen to satisfy $\log 2 <L<\frac{\log 2}{\alpha}$. Next, construct a sequence $B_{j,k}$ of non overlapping open boxes of height $L$ and width $2^j$ as follows:
\begin{itemize}
	\item Take $k_0=\lfloor n^{1-\varepsilon\log 2}\rfloor$, which tends to infinity with $n$ from our assumption on $\varepsilon$. We define the first row of adjacent boxes $\{B_{0,k}\}$, where $k$ ranges from $0$ to $k_0 2^{\lfloor \varepsilon \log n\rfloor}-1$, as a row of adjacent boxes of width $1/2$ and height $L$ of the form $B_{0,k}=(\frac{k}{2},\frac{k+1}{2})\times(0,L)$.

	\item Analogously, for each $j\in\{1,\ldots,\lfloor \varepsilon\log n\rfloor\}$ we construct a row of adjacent boxes $\{B_{j,k}\}$ where now $k$ ranges from $0$ to $k_02^{\lfloor \varepsilon\log n\rfloor-i}$, of width $2^{j-1}$ and height $L$ of the form $B_{j,k}=(2^{j-1}k,2^{j-1}(k+1))\times(jL  ,(j+1)L)$, that is, we construct the row $B_{j,\cdot}$ directly on top of row $j-1$; the only difference being that boxes now have width $2^{j-1}$.
\end{itemize}

\bigskip

\begin{minipage}{0.45\textwidth}
Each $B_{j,k}$ at row $j$ lies below exactly one box $B_{j+1,\lfloor k/2\rfloor}$ from row $j+1$, which we call its \textit{parent}. Conversely, any $B_{j+1,k}$ at row $j+1$ lies on top of exactly two boxes $B_{j,2k}$ and $B_{j,2k+1}$ from row $j$, which we refer to as its \textit{children}. In the picture to the right we can see an example of the construction where $B_{2,0}$ is highlighted as the parent of $B_{1,0}$ and $B_{1,1}$.
 \end{minipage}\hskip0.1in
\begin{minipage}{0.4\textwidth}
\begin{center}
		\begin{tikzpicture}[scale=0.8]
		\draw [draw=black] (-0.5,0) rectangle (7.5,0.7);
	\filldraw [fill=lblue, draw=black] (-0.5,1.4) rectangle (3.5,2.1);
	\filldraw [fill=lred, draw=black] (-0.5,0.7) rectangle (3.5,1.4);
		\foreach \i in {0,...,7}{
			\node at (\i,0.3) {$B_{0,\i}$};
			\draw [draw=black] (\i+0.5,0)--(\i+0.5,0.7);}
		\draw [draw=black] (-0.5,0.7) rectangle (7.5,1.4);
		\foreach \i in {0,...,3}{
			\node at (2*\i+0.5,1.0) {$B_{1,\i}$};
			\draw [draw=black] (2*\i+1.5,0.7)--(2*\i+1.5,1.4);}
		\draw [draw=black] (-0.5,1.4) rectangle (7.5,2.1);
		\foreach \i in {0,1}{
			\node at (4*\i+1.5,1.7) {$B_{2,\i}$};
			\draw [draw=black] (4*\i+3.5,1.4)--(4*\i+3.5,2.1);}
		
		\draw[->,ultra thick] (-0.5,0)--(8.5,0);
		\draw[->,ultra thick] (-0.5,0)--(-0.5,3);
		\node[scale=1.5] at (-0.6,1.75) {$\{$};
		\node at (-0.9,1.75) {\large$L$};
		\end{tikzpicture}

\end{center}
 \end{minipage}

\bigskip

Using this partial order relation between boxes we define a new graph $G$ which will be fundamental in our construction:

\begin{definition}
	Let $\mathcal{B}:=\{B_{j,k}\}_{j,k}$ be as above. We define $G$ as the graph with vertex set $\mathcal{B}$ where any two $B,B'\in\mathcal{B}$ are connected by an edge if either:
	\begin{itemize}
		\item $B$ is the parent of $B'$ (or viceversa), or
		\item $B$ and $B'$ are adjacent boxes at row $\lfloor\varepsilon\log n\rfloor$.
	\end{itemize}
\end{definition}

\medskip

The reason we connect parents to their children is that vertices contained in the corresponding boxes are connected by an edge in ${\bf G}_n$: indeed, take some $(x,h)\in B_{j,k}$ and $(x',h')\in B_{j+1,\lfloor k/2\rfloor}$ and notice that from the definition of the boxes we have $|x-x'|\leq 2^{j}$ and $h,h'\geq jL$ so that 
\[|x-x'|\leq 2^{j}\leq e^{jL}\le \exp\left(\frac{h+h'}{2}\right),\]
and hence $(x,h)$ and $(x',h')$ are neighbors in ${\bf G}_n$. The same reasoning allows us to show that vertices contained in adjacent boxes (that is, in pairs of boxes of the form $B_{j,k}$ and $B_{j,k+1}$) are connected by an edge, since these also satisfy $|x-x'|\leq 2^j$ and $h,h'\geq jL$. We will make use of the latter property only for boxes at row $\lfloor \varepsilon\log n\rfloor$ though.\\

When taking $\lambda$ small, the contact process tends to die out quickly, except on ``good" regions where vertices have an exceptionally large amount of neighbors, enabling the process to survive for a very long time. We will show next that above some fixed row $j_0$, with a large probability the boxes defined above induce large cliques in ${\bf G}_n$ and hence define good regions. Indeed, note that every box induces a clique: observe that for any two vertices $(x,h),(x',h')\in B_{j,k}$ we have  $|x-x'|\leq 2^{j-1}$ and $h,h'\geq jL$, and hence we obtain that $(x,h)$ and $(x',h')$ are neighbors in ${\bf G}_n$, as in the previous argument.

To see that said cliques are large enough, observe that the amount of vertices within $B_{j,k}$ is a Poisson random variable $P_{j,k}$ with parameter
\begin{equation}\label{eq1expext}
\mu_{j}\;:=\;2^{j-1}\int_{jL}^{(j+1)L} \frac{\alpha}{\pi}e^{-\alpha y}dy\;=\;c\, 2^je^{-\alpha j L},
\end{equation}
where $c$ is a positive constant. From our assumption $L<\frac{\log 2}{\alpha}$ it follows that $\mu_j\nearrow\infty$ with $j$, and even further, using a tail bound for Poisson random variables we have that there is some $j_0$ independent of $n$ such that for all $j\geq j_0$,
\begin{equation}\label{eq2expext}
p_j:= \P(P_{j,k}\geq \lambda^{-3})\;\geq\; 1- (e\lambda^3\mu_j)^{\lambda^{-3}}e^{-\mu_j}\;\geq\;1- D e^{-\mu_j/2}, 
\end{equation}
for some $D$ independent of $\mu_j$ and $n$. Since this expression tends to $0$ as $j \to \infty$, 
we conclude that the corresponding cliques at rows with a sufficiently large index are very likely to be large.\\

Say now that a box $B_{j,k}$ is \textit{good} if it contains at least $\lambda^{-3}$ vertices in ${\bf G}_n$. We define a subgraph $\bar{G}\subseteq G$ obtained by
\begin{itemize}
	\item removing from $G$ all vertices $B\in\mathcal{B}$ that are not good, and
	\item removing all connected components from the remaining graph not containing a box at row $\lfloor\varepsilon\log n\rfloor$.
\end{itemize}

\begin{figure}[h!!]
\begin{center}
	\includegraphics[scale=1]{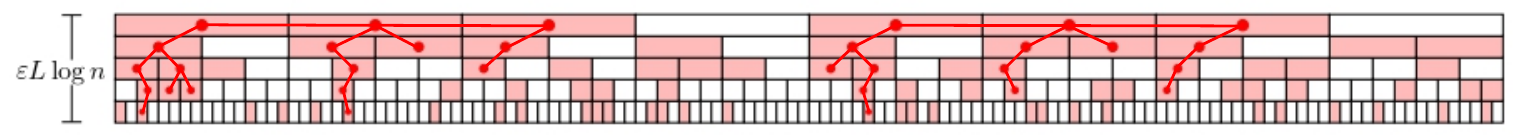}
	\caption{Good boxes are shaded in light red, allowing us to obtain $\bar{G}$ which consists of two connected components in this case.}
\label{figgraph}
\end{center}
\end{figure}

As shown in Figure \ref{figgraph}, the resulting graph $\bar{G}$ consists of a collection of percolated binary trees all having their roots at row $\lfloor\varepsilon\log n\rfloor$, and these roots might or might not be connected. The next result states that with a large probability $\bar{G}$ is not only connected, but also contains a positive fraction of the whole graph ${\bf G}_n$:

\begin{lemma}
	There are some fixed $c>0$ and $\delta,\beta\in(0,1)$ such that
	\[\P(\bar{G}\text{ is connected, and }|\bar{G}|>\delta n)\,\geq\,1-e^{-cn^{\beta}}.\]
\end{lemma}

\begin{proof}
	Notice that from the definition of $G$, the subgraph $\bar{G}$ is connected if and only if all boxes $(B_{\lfloor\varepsilon\log n\rfloor,k})_k$ 
	are good. Now, applying \eqref{eq1expext} and \eqref{eq2expext} for $j=\lfloor\varepsilon\log n\rfloor$ we obtain
	\[\P(B_{j,k}\text{ is good})\;>\;1-D\exp(-Cn^{\varepsilon(\log 2-\alpha L)}),\]
	for all $k$, where $C$ is some positive constant. Since there are at most $k_0=\lfloor n^{1-\varepsilon\log 2}\rfloor$ such boxes, we obtain that for that value of $j$, 
	\begin{equation}\label{eq3expext}
	\P(B_{j,k}\text{ is good }\forall k)\;\geq\;1-Dn^{1-\varepsilon\log 2}\exp(-Cn^{\varepsilon(\log 2-\alpha L)}),
	\end{equation}
	which is already of the form $1-e^{-cn^{\beta}}$. It remains to show that $|\bar{G}|>\delta n$ with a probability of the same order, for which we assume that the event on the left of \eqref{eq3expext} holds. Call $\bar{G}_j$ the set of vertices of $\bar G$ at row $j$ and define the events \[\mathcal{E}_j\;:=\;\{|\bar{G}_{j}|>2(1-f(j))|\bar{G}_{j+1}|\},\]
	with $f(j) = 1/j^2$.  Observe that for any fixed $j$, under $\mathcal{E}_{j},\mathcal{E}_{j+1},\ldots,\mathcal{E}_{\lfloor \varepsilon\log n\rfloor}$ we have
	\begin{equation}\label{eq4expext}
	|\bar{G}_{j}|\,>\,|\bar{G}_{\lfloor\varepsilon\log n\rfloor}|\prod_{\ell=j}^{\lfloor\varepsilon\log n\rfloor-1}2(1-f(\ell))\,>\,c k_02^{\lfloor\varepsilon\log n\rfloor-j-1}\ge c2^{-j-2}n,
	\end{equation}
	where we have used that at row $\lfloor\varepsilon\log n\rfloor$ all boxes are good, and where $c=\prod_{\ell=1}^{\infty}(1-f(\ell))$ is a positive constant. The result then follows if we show that there is some $j_0$ independent of $n$ such that
	\[\P(\mathcal{E}_{j_0},\mathcal{E}_{j_0+1},\ldots,\mathcal{E}_{\lfloor \varepsilon\log n\rfloor})\,\geq\,1-e^{-cn^{\beta}}.\]
	From the construction of $\bar{G}$ and the independence of the events $\{B_{j,k}\text{ is good}\}_{j,k}$ we know that given $|\bar{G}_{j+1}|$ the random variable $|\bar{G}_j|$ follows a binomial distribution with parameters $p_j$ and $2|\bar{G}_{j+1}|$. On the other hand by~\eqref{eq2expext} there is $j_0$ large such that for $j>j_0$ we have $1-f(j)< p_j$, and thus using Chernoff's bound we obtain
	\[\P\big(\mathcal{E}_j\;\big|\;|\bar{G}_{j+1}|\big)\,\geq\,1-\exp\left(-\frac{\mu_jf(j)|\bar{G}_{j+1}|}{2}\right),\]
	which is increasing in $|\bar G_{j+1}|$. From the discussion leading to \eqref{eq4expext}, there is a constant $\bar{c}>0$, such that
	\[\P(\mathcal{E}_{j}\,|\,\mathcal{E}_{j+1},\ldots,\mathcal{E}_{\lfloor \varepsilon\log n\rfloor})\,\geq\,1-\exp\left(-\bar{c}n2^{-j}\mu_jf(j)\right)\,\geq\,1-\exp\left(-\bar{c}n^{1-\alpha\varepsilon L}f(\varepsilon \log n)\right),\]
	where we used that $2^{-j}\mu_j f(j)$ is decreasing in $j$. Finally we conclude that
	\[\P(\mathcal{E}_{j_0},\mathcal{E}_{j_0+1},\ldots,\mathcal{E}_{\lfloor \varepsilon\log n\rfloor})\,\geq\,1-\lfloor\varepsilon \log n\rfloor e^{-\bar{c}n^{1-\alpha\varepsilon L}f(\varepsilon\log n)},\]
	which is larger than $1-e^{-c n^{\beta}}$, for any $\beta<1-\alpha\varepsilon L$, and some $c>0$ depending on $\beta$.	
\end{proof}
	\medskip

We are now ready to give the proof of Theorem~\ref{theo.exp}: take a realization of ${\bf G}_n$ such that $\bar{G}$ is connected and $|\bar{G}|>\delta n$, and construct the subgraph $\bar{{\bf G}}_n\subseteq {\bf G}_n$ with vertex set ${\bf G}_n\cap \cup_{B\in\bar{G}}B$ as follows:
\begin{enumerate}
	\item For each $B\in\bar{G}$ choose an arbitrary vertex $v_B$ and let all the remaining vertices in $B$ to be connected by an edge to $v_B$ (and to no other vertex),
	\item Add the edge $\{v_B,v_{B'}\}$ to $\bar{{\bf G}}_n$ if and only if $B\sim B'$ in $\bar{G}$.
\end{enumerate}

It follows that $\bar{{\bf G}}_n$ is composed of at least $\delta n$ stars of size no smaller than $\lambda^{-3}$, which are connected by their centers. For such a structure it was already proved in \cite{MMVY} that the infection starting from the fully infected configuration satisfies

$$\P[\uptau_{\bar{{\bf G}}_n}>e^{cn}] > 1-e^{-cn}$$

for some $c>0$, and the result follows.
\begin{flushright}
	$\blacksquare$
\end{flushright}

We now provide the proof of Proposition \ref{badgn} by constructing the bad event $A_n$ as follows: Choose some $a\in(\frac{1}{2},1)$ and some $\varepsilon\in(0,1)$ with $a+\varepsilon<1$. Take now an ordered sequence $\{x_k\}$ of evenly spaced points in $[-\frac{\pi}{2}n,\frac{\pi}{2}n]$ with distance equal to $n^{1-a}$. Observe that $k$ ranges from $1$ to $\lfloor \pi n^a\rfloor$. We use these points to divide the space $[-\frac{\pi}{2}n,\frac{\pi}{2}n]\times[0,2\log n)]$ into the sets 
\[\begin{array}{rl}
B_0&=\{(x,h),\,h\geq \varepsilon \log n\}\\[7pt]
B_k&=\{(x,h),\,|x-x_k|\leq \frac{1}{2}n^\varepsilon\text{ and }h<\varepsilon \log n\}\\[7pt]
C_k&=\{(x,h),\,x_k+\frac{1}{2}n^\varepsilon<x<x_{k+1}-\frac{1}{2}n^\varepsilon\text{ and }h<\varepsilon \log n\}, 
\end{array}\]
which are well defined because $\varepsilon<1-a$. It follows directly from the definition of $\mu$ that 
$\mu(B_0)\le  n^{1-\alpha\varepsilon}$,  
and for all $k\geq 1$, $\mu(B_k)\leq n^\varepsilon$  and $\mu(C_k)\leq n^{1-a}$. 
As a result, there are positive constants $c$ and $C$, such that
\[\begin{array}{rlll}
\P(B_0\cap{\bf G}_n=\emptyset)&\geq & e^{-n^{1-\alpha\varepsilon}}\\[7pt]
\P(B_k\cap{\bf G}_n=\emptyset)&\geq & e^{-n^{\varepsilon}}& \text{ for all }k\geq1\\[7pt]
\P(|C_k\cap{\bf G}_n|\geq C n^{1-a})&\leq  & e^{-cn^{1-a}}& \text{ for all }k\geq1.
\end{array}\]
Define $A_n$ as the event in which there are no vertices in any of the $B_k$ (including $k=0$), and in every $C_k$ there are at most $C n^{1-a}$ vertices. Using that all sets correspond to disjoint areas, we obtain
\[\P(A_n)\,\geq\,e^{-n^{1-\alpha\varepsilon}}\left(e^{- n^{\varepsilon}}\right)^{\pi n^a}\left(1-e^{-cn^{1-a}}\right)^{\pi n^a}\geq\frac{1}{2}e^{-\pi(n^{1-\alpha\varepsilon}+n^{a+\varepsilon})}=e^{-n^{\beta}},\]
for some $\beta>0$.

Now observe that if we take $v=(x,h)$ and $v'=(x',h')$ belonging to different $C_k$ we necessarily have $|x-x'|>n^\varepsilon$, since there is at least one set $B_k$  between them, and also $e^{\frac{h+h'}{2}}<n^{\varepsilon}$, so that $v$ and $v'$ cannot be neighbors in ${\bf G}_n$. 
It follows that on $A_n$ the graph ${\bf G}_n$ is composed of connected components each of size at most $C n^{1-a}$. 
As shown in \cite[Lemma 2.3]{SV}, this entails that the expected extinction time on each of these connected components is at most $e^{C'n^{2-2a}}$, for some other constant $C'>0$. 
Since there are at most $n$ such components, we finally deduce 
\[\E(\tau_{{\bf G}_n}\cdot {\bf 1}_{A_n}) \leq n\exp(C'n^{2-2a}),\]
and the result follows from the assumption $a>1/2$.


\section{Survival probability: lower bounds}\label{sec.lower}
In this section we prove the lower bounds in Theorem~\ref{thm:gammafraction}. 
We give two different strategies that show that the contact process survives for a long time. In a nutshell, in the case $\alpha \in (\tfrac12,\tfrac34]$ the strategy of surviving corresponds to finding a neighbor of the root of sufficiently high degree, from which the infection will then pass over to vertices of even higher degree, and thus surviving an infinite amount of time. In the case $\alpha \in (\tfrac34, 1)$ the strategy is different: a neighbor at a high level is infected, but all its neighbors of low degree are needed to infect a vertex of even higher degree (using Lemma~\ref{contactNstar}). We make this more precise in the next two subsections.

\subsection{Case~$\alpha \in (\tfrac12,\tfrac34]$}
The goal is to prove the following lemma:
\begin{lemma}\label{lower1234}
Let $\alpha \in (\tfrac12, \tfrac34]$. Then 
$$
\gamma(\lambda) > c\lambda^{\frac{1}{2-2\alpha}},
$$
for some sufficiently small constant $c=c(\alpha)$ depending on $\alpha$ only. 
\end{lemma}

\begin{proof}
We consider the contact process~$(\xi_t)$ on~${\bf G}_\infty$ started from a single infection at the root,~$\xi_0 = \mathds{1}_{\{o\}}$. Let~$h_*:= \frac{1}{1-\alpha}\log(C/\lambda),$ with~$C$ some large constant to be chosen later.  Let~$E_0$ denote the event that the height of the root is below~$h_*$. Also define~$\tau_0 := 0$ and let~$\tau_0'$ be the first recovery time at~$o$.

Let~$E_1$ denote the event that~$E_0$ occurs, and that~$o$ has a neighbor~$\hat{v}_1 \in \mathbb{R}\times [h_*,h_*+1)$, and there is a transmission from~$o$ to~$\hat{v}_1$ at a time~$\tau_1 \in [\tau_0,\tau_0')$. Recursively, assume that events~$E_0 \supset \ldots \supset E_k$ are defined, that they only involve information on the portion of the graph contained in~$\mathbb{R}\times [0,h_*+k)$, and that~$E_k$ involves a vertex~$\hat{v}_k \in \mathbb{R}\times [h_*+k-1,h_*+k)$ receiving the infection at a time~$\tau_k$. On~$E_k$, let~$\tau_k'$ denote the first recovery time at~$\hat{v}_k$ after~$\tau_k$. Then, let~$E_{k+1}$ be the event that~$E_k$ occurs, and additionally~$\hat{v}_k$ has a neighbor~$\hat{v}_{k+1} \in \mathbb{R}\times [h_*+k,h_*+k+1)$, and there is a transmission from~$\hat{v}_k$ to~$\hat{v}_{k+1}$ at a time~$\tau_{k+1} \in [\tau_k,\tau_k')$. Clearly, if~$\cap_{k\ge 0} E_k$ occurs, then~$\xi_t \neq \varnothing$ for all~$t$, that is, the process survives.

We now give lower bounds to the  probabilities of these events, starting with
$$\mathbb{P}(E_0) = \int_0^{h_*}\alpha e^{-\alpha h}\;\mathrm{d}h > \frac12,$$
if~$\lambda$ is small (and hence~$h_*$ is large). Next, denoting the height of~$o$ by~$h_o$, the number~$N_0$ of neighbors of~$o$ in~$\mathbb{R}\times[h_*,h_*+1)$ follows a Poisson distribution with parameter
$$\int_{h_*}^{h_*+1} \exp\left\{\frac{h_o+h}{2} -\alpha h\right\}\mathrm{d}h \ge \beta_0 := \int_{h_*}^{h_*+1}\exp\left\{\left(\frac12-\alpha \right)h\right\}\mathrm{d}h \ge c\left(\frac{\lambda}{C}\right)^{\frac{2\alpha - 1}{2-2\alpha}},$$
with~$c$ some positive constant depending only on~$\alpha$ (which may change from line to line). Hence, 
$$\mathbb{P}(E_1 \mid E_0) \ge \frac{\lambda}{1+\lambda}\cdot \mathbb{P}(N_0 \ge 1 \mid E_0) \ge c\lambda \cdot \beta_0 \ge c\cdot C^{-\frac{2\alpha - 1}{2-2\alpha}}\cdot \lambda^{\frac{1}{2-2\alpha}}, $$
where we used that~$\lambda$ is small, so that~$1+\lambda < 2$ and~$e^{-\beta_0} > \frac12$.

Next, on~$E_k$, let~$N_k$ denote the number of neighbors of~$\hat{v}_k$ on~$\mathbb{R}\times [h_*+k,h_*+k+1)$. We have that (conditioned on~$E_k$) the law of~$N_k$ is Poisson with parameter larger than
$$\beta_k := \int_{h_* +k}^{h_*+k +1} \exp\left\{ \frac{h_*+k-1 +h}{2} -\alpha h \right\}\mathrm{d}h \ge c  \exp\left\{(1-\alpha)(h_*+k)\right\} = \frac{cC}{\lambda}\cdot e^{(1-\alpha)k}.$$
Then, by the strong Markov property,
$$\mathbb{P}(E_{k+1}^c\mid E_k) = \mathbb{E}\left[\left. \frac{1}{1+\lambda N_k} \right| E_k\right]\le \mathbb{P}(N_k \le \beta_k/2\mid E_k) + \frac{1}{1+\lambda \beta_k/2}.$$
By a Chernoff bound we have,~$\mathbb{P}(N_k \le \beta_k/2\mid E_k) \ll \beta_k^{-1}$, so we obtain
$$\mathbb{P}(E_{k+1}^c \mid E_k) \le \frac{c}{\lambda \beta_k} \le \frac{c}{C}\cdot e^{-(1-\alpha)k} .$$

Putting these bounds together we have
$$\mathbb{P}\left(\cap_{k \ge 0}E_k\right) \ge \frac12 \cdot cC^{-\frac{2\alpha -1}{2-2\alpha}}\cdot \lambda^{\frac{1}{2-2\alpha}}\cdot \prod_{k\ge 1}\left(1-\frac{c}{C}\cdot e^{-(1-\alpha)k}\right).$$
Recalling that~$c$ depends only on~$\alpha$, and choosing~$C>c$, so that the infinite product on the right-hand side is positive, the proof is complete.
\end{proof}

\subsection{Case~$\alpha \in (\tfrac 34,1)$}
The goal is to prove the following lemma:
\begin{lemma}\label{lower341}
Let $\alpha \in (\tfrac34, 1)$. Then 
$$
\gamma(\lambda) > c \cdot \frac{\lambda^{4\alpha - 1}}{\log(1/\lambda)^{2\alpha - 1}},
$$
for some sufficiently small constant $c=c(\alpha)$ depending on $\alpha$ only.
\end{lemma}
Before we prove this, we state an auxiliary result.  Recall the definition of the graph~$\mathbb{L}_d$ from Lemma~\ref{contactNstar}, consisting of a ``half-line of stars''.
\begin{lemma}\label{lem:contactNstar_hyp} Let~$C>0$ and~$d=C\log(1/\lambda)/\lambda^2$ be as in Lemma~\ref{contactNstar}, and let~$h_{**}:= 2\log(2d)$. Let~$v = (x_v,h_v) \in \mathbb{H}$ with~$h_v > h_{**}$, and let~${\bf G}_\infty^v$ be the random hyperbolic graph with a vertex artificially added at~$v$. Then, with probability tending to one as~$\lambda \to 0$,~${\bf G}^v_\infty$ has a subgraph isomorphic to~$\mathbb{L}_d$, entirely contained in~$[x_v,\infty) \times [0,\infty)$, and so that~$v$ plays the role of the center of the first star of the half-line.
\end{lemma}
Let us now show how this lemma allows us to prove our lower bound on the survival probability.

\begin{proof}[Proof of Lemma~\ref{lower341}]
As before, we start a contact process on~${\bf G}_\infty$ with only the root~$o$ infected. Writing~$o = (x_o,h_o)$, we first consider the event~$E_0$ that the root has at least one neighbor in~$[x_o,\infty) \times [h_{**},\infty)$. On this event, we let~$\hat{v} = (x_{\hat{v}},h_{\hat{v}})$ denote the neighbor of~$o$ on~$[x_o,\infty) \times [h_{**},\infty)$ such that~$x_{\hat{v}}$ is minimal. Furthermore, let~$E_1$ be the event that~$E_0$ occurs and there is a transmission from~$o$ to~$\hat{v}$ before the first recovery at~$o$. We then have
$$\mathbb{P}(E_1) \ge \frac{c\lambda}{1+\lambda} \cdot  e^{(\frac12-\alpha)h_{**}}  \ge  \frac{c\lambda^{4\alpha-1}}{\log(1/\lambda)^{2\alpha- 1}},$$
for some positive constant $c$ that only depends on $\alpha$. Conditioned on~$E_1$, since the graph on~$[x_{\hat{v}},\infty) \times [0,\infty)$ is still unrevealed, and by Lemma~\ref{lem:contactNstar_hyp}, with probability larger than~$\frac12$ (if~$\lambda$ is small),~$\hat{v}$ is the first star in a copy of~$\mathbb{L}_d$ entirely contained in~$[x_{\hat{v}},\infty) \times [0,\infty)$. Conditioned on this subgraph being present, the infection then survives with a probability bounded from below by a positive constant, uniformly in~$\lambda$, by Lemma~\ref{contactNstar}.
\end{proof}
It remains to prove the auxiliary result:
\begin{proof}[Proof of Lemma~\ref{lem:contactNstar_hyp}]
By invariance of the point process under horizontal translations, it suffices to treat the case~$x_v = 0$. We define~$H_k:= h_{**}+k$ for~$k\ge 0$; also let
$$\ell_k:= {e^{H_k-2}},\; k \ge 0 ,\qquad L_0 := 0,\quad L_k:= \sum_{j=0}^{k-1} \ell_j,\qquad k \ge 1.$$
Next, define the boxes
$$S_k := [L_{k},\;L_{k+1}) \times [H_{k},\;\infty),\quad S_k':= [L_{k},\;L_{k+1}) \times [0,1],\qquad k \ge 0; $$
note that they are all disjoint. We now state and prove two claims about these boxes.
\begin{claim}
\label{cl:box1} Let~$k \in \mathbb{N}$ and condition on~$a = (x_a,h_a) \in S_k$ being  a vertex of~${\bf G}_\infty^v$. Then,~$a$ has a neighbor in~$S_{k+1}$ with probability larger than~$1-\exp\{-\frac{1}{\alpha}e^{(1-\alpha)(h_{**}+k)-1}\}.$
\end{claim}
\begin{proof}
First note that any vertex~$b = (x_b,h_b) \in S_{k+1}$ is necessarily a neighbor of~$a$, since
$$|x_a-x_b| \le \ell_k+\ell_{k+1} = {e^{H_k-2} + e^{H_k-1}} \le e^{H_k}\le e^{\frac{h_a+h_b}{2}}.$$
Hence, we only need to estimate the probability that~$S_{k+1}$ has no vertices. Since the number of vertices in~$S_{k+1}$ is Poisson with parameter at least
$$\ell_{k+1}\cdot \int_{H_{k+1}}^\infty e^{-\alpha h}\mathrm{d}h =  \frac{1}{\alpha}\cdot e^{(1-\alpha)(h_{**}+k)-1},$$
the result follows.
\end{proof}
\begin{claim} The following holds for~$\lambda$ small enough: let~$k \in \mathbb{N}$ and condition on~$a = (x_a,h_a) \in S_k$ being  a vertex of~${\bf G}_\infty^v$. Then,~$a$ has at least~$d$ neighbors in~$S_{k}'$ with probability larger than~$1-\exp\{-cde^{k/2}\}$, for some~$c > 0$ that does not depend on~$\lambda$ or~$k$.
\end{claim}
\begin{proof}
First note that, since~$e^{H_k/2} \ll \frac12 e^{H_k - 2} = \frac12 \ell_k$ for any~$k$ if~$\lambda$ is small, at least one of the boxes
$$[x_a-e^{H_k/2},\;x_a]\times [0,1]\quad \text{and}\quad [x_a,\;x_a + e^{H_k/2}]\times [0,1]$$
is contained in~$S_k'$. Moreover, any vertex in these two boxes is connected by an edge to~$a$, since~$h_a \ge H_k$. The number of vertices inside any of the two boxes is Poisson with parameter
$$e^{H_k/2}\int_0^1 e^{-\alpha h}\mathrm{d}h = \frac{1-e^{-\alpha}}{\alpha}\cdot e^{H_k/2}= 2d\cdot e^{k/2}.$$
By a Chernoff bound, such a Poisson random variable is larger than~$d$ with probability larger than~$1-\exp\{-cde^{k/2}\}$ for some universal constant~$c > 0$, completing the proof.
\end{proof}
Now, combining the two claims and independence of the point process in disjoint pairs of boxes, the probability that we can find a sequence~$v_0= v, v_1,v_2,\ldots$ so that for every~$k$ we have~$v_k \in S_k$,~$v_k \sim v_{k+1}$ and~$v_k$ has at least~$d$ neighbors in~$S_k'$, is larger than
$$1 - \prod_{k=0}^\infty \left(1-\exp\left\{-\frac{1}{\alpha}e^{(1-\alpha)(h_{**}+k)-1} \right\}\right)\cdot (1-\exp\{-cde^{k/2}\}),$$
which can be made as close to~$1$ as desired by taking~$\lambda$ small, since~$h_{**} \to \infty$ as~$\lambda \to 0$.
\end{proof}

\section{Survival probability: upper bounds}\label{sec.upper}
We prove here the upper bounds in Theorem~\ref{thm:gammafraction}. 
We start with a general result (see Lemma~\ref{lem:super} below) regarding the existence of infection paths.

\subsection{Infection paths and ordered traces}
Given a graph~$G= (V,E)$, we define~$\Gamma_\infty = \Gamma_\infty(G)$ as the set of all finite and infinite sequences of the form~$(\gamma(0),\gamma(1),\ldots)$ with $\gamma(0),\gamma(1),\ldots \in V$ and~$\gamma(i)\sim \gamma(i+1)$ for each~$i$. Elements of~$\Gamma_\infty$ are called \textit{vertex paths};  the \textit{length} of a finite vertex path~$\gamma = (\gamma(0)\ldots,\gamma(k))$ is defined as~$|\gamma|:= k$; in case~$\gamma$ is infinite, we set~$|\gamma| = \infty$.

Assume given a graphical construction for the contact process~$(\xi_t)_{t \ge 0}$ with some rate~$\lambda > 0$ on~$G$. Recall the definition of infection paths from Section~\ref{ss:contact}. Given an infection path~$g:I\to V$, where~$I \subset \mathbb{R}$ is an interval, we say that the \textit{ordered trace} of~$g$ is the vertex path~$\gamma_g = (\gamma_g(0),\ldots) \in \Gamma_\infty$ obtained by setting~$\gamma_g(0)$ as the vertex where~$g$ starts,~$g((\inf I)^+)$, and letting the subsequent vertices of~$\gamma_g$ be the vertices visited by~$g$ in order.

\begin{lemma}\label{lem:super}
Assume~$\lambda < \frac12$. Given~$\gamma\in \Gamma_\infty$, the probability that there exists~$t \ge 0$ and an infection path~$g:[0,t]\to V$ having~$\gamma$ as its ordered trace is at most~$(2\lambda)^{|\gamma|}$.
\end{lemma}
\begin{proof}
Fix~$\gamma \in \Gamma_\infty$. For each~$t \ge 0$, define~$X_t$ as the largest value of~$i \in \{0,\ldots,|\gamma|\}$ such that there is an infection path~$g:[0,t] \to V$ with~$g(0) = \gamma(0)$ and ordered trace~$\gamma_g = (\gamma(0),\ldots,\gamma(i))$ (let~$X_t = -\infty$ in case no such~$i$ exists). Let
$$\tau = \inf\{t: X_t  \in \{-\infty,|\gamma|\}\},$$
and note that the event described in the statement of the lemma occurs if and only if~$X_\tau = |\gamma|$. Next, define
$$M_t = (2\lambda)^{-X_t}, \;t \ge 0,$$
so that~$M_\tau = 0$ when~$X_\tau = -\infty$.
We claim that~$(M_{\tau \wedge t})_{t \ge 0}$ is a supermartingale with respect to the natural filtration~$(\mathcal{F}_t)_{t \ge 0}$ of the Poisson processes in the graphical construction. To see this, note that, on~$\{\tau > t\}$, 
\begin{align*}\left.\frac{\mathrm{d}}{\mathrm{d}s}\mathbb{E}[M_{t+s}\mid \mathcal{F}_t]\right\vert_{s = 0} &=  \left( (2\lambda)^{-(X_t-1)} - (2\lambda)^{-X_t} \right)+\lambda \left((2\lambda)^{-(X_t + 1)} - (2\lambda)^{-X_t}\right)\\[.2cm]
&= (2\lambda)^{-X_t}\cdot \left(2\lambda - 1 + \frac12 - \lambda\right) < 0,
\end{align*} 
assuming~$\lambda < \frac12$. Now, the optional stopping theorem gives
$$1=\mathbb{E}[M_0] \ge \mathbb{E}[M_\tau] \ge (2\lambda)^{-|\gamma|}\cdot \mathbb{P}(M_\tau = (2\lambda)^{-|\gamma|}) = (2\lambda)^{-|\gamma|}\cdot \mathbb{P}(X_\tau = |\gamma|),$$
completing the proof.
\end{proof}
In what follows, we write, for~$h > 0$ and~$d > 0$,
\begin{equation}\label{eq:d_and_h}D(h) = \frac{1}{\alpha -\frac12}\cdot e^{h/2},\quad H(d) = D^{-1}(d) = 2\log\left(\left(\alpha - \frac12\right)\cdot  d\right).\end{equation}
Note that $D(h)$ corresponds to the expected degree of a vertex at height $h$; the value~$H(d)$ should be thought of as a height compatible with degree~$d$.

\subsection{Regime $\alpha \in (\tfrac12, \tfrac34]$}
The goal of this section is to prove the following proposition:
\begin{proposition}\label{upper1234}
Let $\alpha \in (\tfrac12, \tfrac34]$. Then 
$$
\gamma(\lambda) < C\lambda^{\frac{1}{2-2\alpha}}, 
$$
for some sufficiently large constant $C=C(\alpha)$ depending on $\alpha$ only. 
\end{proposition}
\begin{proof}
Let $d_0=c\lambda^{-\frac{1}{2-2\alpha}}$ for a sufficiently small constant $c=c(\alpha) > 0$. Call a vertex to be \textit{red} if its height is at least $h_0=H(d_0)$ (in other words its expected degree is at least $d_0$), and all others \textit{blue}. Starting from $o$ (that was artificially added), we say we exit the $k$-th neighborhood of $({\bf G}_\infty, o)$, if either the infection spreads through a path of all blue vertices of length $k$, or if a red vertex at distance less than $k$ from $o$ is infected, or if a blue vertex already appearing on a blue path becomes re-infected (we do not claim that the vertex healed in the meantime, we just say that there was another infection that took place, that is, another transmission arrow in the graphical construction).  We will show that for $k=\log(1/\lambda)$, the probability to exit the $k$-th neighborhood is at most $C\lambda^{\frac{1}{2-2\alpha}}$, thus proving the desired statement. 
 We define the following events:
 $$
 E_1=\{ \mbox{o is red}\}
 $$
\begin{align*}
 E_2=&\{ \mbox{o is blue, there exists a path of length } 0 \le j < k  \mbox{ of (all different) infected blue vertices,} \\ 
  &\mbox{ followed by a red vertex that is infected} \}
 \end{align*}
 \begin{align*}
 E_3=&\{ \mbox{o is blue, there exists a path of (all different) blue vertices of length } k \\
 & \mbox{through which the infection travels}\}
 \end{align*}
 \begin{align*}
 E_4=&\{\mbox{o is blue, there exists a path of (all different) blue vertices of length } 1 \le j < k \\
 &\mbox{ followed by a blue vertex that appeared previously on the path that is infected again} \}
 \end{align*}
 It is clear that if none of $E_1, E_2, E_3, E_4$ happens then the infection does not survive.

 For $E_1$, the probability that $o$ is red is ($C=C(\alpha)$ is a sufficiently large constant that changes from line to line), 
$$
\alpha \int_{h \ge h_0} e^{-\alpha h} dh = e^{- \alpha h_0} = C\lambda^{\frac{2\alpha}{2-2\alpha}}<  C\lambda^{\frac{1}{2-2\alpha}}.
$$
Next, consider a path of length $1 \le j < k$, of (all different) blue vertices followed by a red vertex, through which the infection travels. 
For $j+1$ (ordered) distinct vertices $o, x_1, \ldots, x_{j}$,  let $F_2(o, x_1, \ldots, x_{j})$ be the indicator function for vertex $x_1$ being  infected by $o$; for $i=2, \ldots, j-1$, $x_i$ being blue and being infected by $x_{i-1}$, and finally, $x_{j}$ being red and being infected by $x_{j-1}$. 
By the multivariate Mecke formula (see for example \cite[Theorem 4.4]{Last-Penrose}) and Lemma~\ref{lem:super}, we have 
\begin{align*}
& \mathbb{E} \left(\sum_{o, x_1, \ldots, x_{j}}^{\neq}(F_2(o, x_1, \ldots, x_{j}))\right)  \\
 & \le   (C\lambda)^j \int_{h < h_0}\int_{h_1 < h_0} \cdots \int_{h_{j-1} < h_0} \int_{h' \ge h_0}e^{(1-\alpha)(h_1+\ldots+h_{j-1})}e^{(\frac12-\alpha)(h+h')}\, dh' dh_{j-1} \ldots dh_1dh\\
 & \le (C\lambda)^j e^{(1-\alpha)(j-1)h_0+(\frac 12-\alpha)h_0}\le (Cc^{2-2\alpha})^j \cdot \lambda^{1-\frac{1-2\alpha}{2-2\alpha}},
 \end{align*}
where the sum is over $(j+1)$-tuples of vertices being all different; indeed, the Mecke formula gives the desired integral representation for the expected number of vertices in the desired region, and since conditional under having points at certain locations the probability of 
having infections is bounded by Lemma~\ref{lem:super},  the expected number of infection paths is the product of the existence of paths together with the indicator variable of having an infection throughout the path, giving the desired formula.
Therefore, 
\begin{align*}
 \mathbb{E}\left(\sum_{j=1}^{k-1} \sum_{o, x_1, \ldots, x_{j}}^{\neq}(F_2(o, x_1, \ldots, x_{j}))\right)  
 \le  \lambda^{1-\frac{1-2\alpha}{2-2\alpha}} \cdot \sum_{j=1}^{k-1}(Cc^{2-2\alpha})^j \le C\lambda^{\frac{1}{2-2\alpha}}, 
\end{align*}
where in the the last inequality we assumed $c$ sufficiently small so that the sum is convergent. Note that in order for $E_2$ to hold, there must exist $1 \le j < k$ and $o, x_1, \ldots, x_{j-1}, x_{j}$ so that $F_2(o, \ldots, x_{j-1}, x_{j})=1$, and hence, by a union bound we have the desired upper bound on the probability of $E_2$.

By the same argument, for $E_3$, the probability of having a path of (all different) blue vertices of length $k=\log(1/\lambda)$ through which the infection travels is at most
$$
(C\lambda)^k \int_{h \le h_0}\int_{h_1 \le h_0} \cdots \int_{h_k \le h_0} e^{(1-\alpha) (h_1+\ldots+h_{k-1})}e^{(\frac12-\alpha)(h+h_k)} \le \lambda(Cc^{2-2\alpha})^{k-1} \le C\lambda^{\frac{1}{2-2\alpha}},
$$
where we assumed again $c$ sufficiently small, and used $\alpha>1/2$ for the last inequality.

Finally, for the probability that $o$ is blue, and that there is a path of (all different) blue vertices of length $1 \le j < k$ through which the infection travels, followed by a blue vertex that appeared previously on the path, observe that for the last vertex that is repeated, there are $j+1$ choices to choose the vertex. Since this vertex is already there, there is no additional factor corresponding to the intensity of having a vertex there, there is however an additional factor $\lambda$ for re-infecting the previously appeared vertex. Let 
$F_4(o, x_1, \ldots, x_{j}, x_r)$ be the indicator function for vertex $x_1$ being infected by $o$; for $i=2, \ldots, j-1$, $x_i$ being blue and being infected by $x_{i-1}$ (all vertices up to $x_j$ being distinct), and finally, $x_{r}$ is infected by $x_j$, where $x_r$ is a repeated vertex (for which there are $j+1$ choices). Once again by the multivariate Mecke formula we have (summing over all tuples of vertices where only the last vertex is repeated, all others being distinct), 
\begin{align*}
&\mathbb{E} \left(\sum_{j=1}^{k-1}\sum_{o, x_1, \ldots, x_{j},x_r} (F_4(o, x_1, \ldots, x_{j},x_r))\right)  \\
& \le \sum_{j=1}^{k-1} (C\lambda)^{j+1}(j+1) \int_{h \le h_0}\int_{h_1 \le h_0} \cdots \int_{h_j \le h_0} e^{(1-\alpha) (h_1+\ldots+h_{j-1})}e^{(\frac12-\alpha)(h+h_j)}\, dh_j\ldots dh_1dh \\
 & \le \sum_{j=1}^{k-1} (j+1)\lambda^2 (Cc^{2-2\alpha})^{j-1} \le C \lambda^2  \le C\lambda^{\frac{1}{2-2\alpha}},
\end{align*}
where the sum is over tuples of vertices with $(o, x_1, \ldots, x_j)$ being all different and $x_r \in \{o, x_1, \ldots, x_j\}$, and where we used for the last inequality that $\alpha \le 3/4$. By taking a union bound over the probabilities of all events $E_1, E_2, E_3, E_4$, the proof is finished.
\end{proof}

\subsection{Regime $\alpha \in (\tfrac34, 1)$}
{The goal of this section is to prove the following proposition.}
\begin{proposition}\label{upper341}
Let $\alpha \in (\tfrac34, 1)$. Then 
$$
\gamma(\lambda) < C \cdot \frac{\lambda^{4\alpha - 1}}{\log(1/\lambda)^{2\alpha - 1}}
$$
for some sufficiently large constant $C=C(\alpha)$ depending on $\alpha$ only.
\end{proposition}

Before turning to the proof of this result, we need to make a detour, with several definitions and intermediate results. To justify why this is needed, we first point out that, in the upper bound for the case~$\alpha \in (\tfrac12,\tfrac34]$, we did not really have to deal with the infection spreading from vertices of degree above~$d_0 = c\lambda^{-\frac{1}{2-2\alpha}} \ll \lambda^{-2}$: such vertices were labelled red there, and the probability of their ever becoming infected was already small for the purposes of our upper bound. For the present case~$\alpha \in (\tfrac34,1)$, however, the event that the root has a neighbor of degree around~$\lambda^{-2}$, and infects this neighbor, has probability of larger order than what we hope to achieve with our union bound. Hence, we need to include this event in our proof, and go further by saying that even if it happens, the infection has small chance of surviving thereafter. To do so, we need to develop tools to argue that the infection does not travel far even if it starts from a vertex whose degree is around~$\lambda^{-2}$; around these vertices, multiple re-infections are likely to occur. 

We fix a rooted graph~$(G=(V,E),o)$, and consider the contact process~$(\xi_t)_{t \ge 0}$ on~$G$ started from~$\xi_0 = \{o\}$ (in all that follows, this initial configuration will be assumed). Given a vertex~$u \in V$, we say that~$(\xi_t)$ \textit{is thin on }$u$ in the event that there is no infection path~$g:[0,t]\to V$ for some~$t \ge 0$ with~$g(0)=o$ and such that~$u$ appears more than once in the ordered trace of~$g$. We say that~$(\xi_t)$ is thin on a set~$V' \subset V$ in the event that~$(\xi_t)$ is thin on every vertex of~$V'$.
\begin{lemma}\label{lem:Afin}
If~$V_0 \subset V$ is finite, then on the event that~$(\xi_t)$ is thin on~$(V_0)^c$, it almost surely dies out, that is, almost surely there is~$t \ge 0$ such that~$\xi_t = \varnothing$.
\end{lemma}
\begin{proof}
For~$t \ge 0$, let~$E_t$ be the event that~$\xi_t \neq \varnothing$ and the ordered trace~$\gamma_g$ of any infection path~$g:[0,s] \to V$ with~$g(0) =o$ and~$s \le t$ visits each vertex of~$(V_0)^c$ at most once. Using the finiteness of~$V_0$, and making a finite number of prescriptions on Poisson processes on the graphical construction, it is easy to see that~$\mathbb{P}(E_{t+1}) \le \sigma(\lambda,V_0) \cdot \mathbb{P}(E_{t})$ for some~$\sigma(\lambda,V_0) < 1$ (it suffices for example to extend an existing infection path~$g:[0,t]\to V$ by imposing that in the time interval~$(t,t+1)$, it reaches some~$u \in (V_0)^c$, then from there jumps to a neighbour of~$u$, then to~$u$ again). We then have
$$\mathbb{P}\left(\xi_t \neq \varnothing \; \forall t,\; (\xi_t) \text{ is thin on }(V_0)^c\right) = \lim_{t \to \infty}\mathbb{P}(E_t) = 0.$$
\end{proof}
Before stating the next result, we will need to define some subsets of~$\Gamma_\infty$. We fix a set~$A \subset V$ with~$o \notin A$, and define 
\begin{align}\nonumber &\Gamma_A^k := \left\{\begin{array}{l}(\gamma(0),\ldots,\gamma(k)) \in \Gamma_\infty:\;\gamma(0) = o,\\[.2cm]\gamma(0),\ldots,\gamma(k-1)\text{ are distinct and not in }A,\\[.2cm] \gamma(k) \in A\end{array}\right\},\; k \ge 1\\[.2cm]
&\Gamma_{A,*}^k := \left\{\begin{array}{l} (\gamma(0),\ldots,\gamma(k)) \in \Gamma_\infty:\;\gamma(0) = o,\\[.2cm]\gamma(0),\ldots,\gamma(k-1)\text{ are distinct and not in }A,\\[.2cm]\gamma(k) \in \{\gamma(0),\ldots,\gamma(k-1)\}\end{array} \right\},\; k \ge 3, \nonumber \\[.2cm]
&\Gamma_{A} := \cup_{k \ge 1}\;\Gamma_A^k,\quad \Gamma_{A,*} = \cup_{k \ge 3}\;\Gamma_{A,*}^k. \label{eq:def_gammas}  
\end{align}
{The role of $A$ will become clear in the sequel, but the intuition is that in the hyperbolic graph setting $A$ is a set of dangerous vertices (typically vertices above a certain height and thus of high degree) whose infection should rather be avoided, as otherwise the infection goes on for too long. Nevertheless, the following lemma holds in a more general setup:}

\begin{lemma} \label{lem:star_lem} There exists~$c > 0$ such that, for any~$\lambda < \frac12$, the following holds.
Let~$G,o,A$ be as above, and let~$(\xi_t)_{t\ge 0}$ be the contact process with parameter~$\lambda$ on~$G$ with~$\xi_0 = \{o\}$. Then,
$$\mathbb{P}\left(\xi_t \neq \varnothing \;\forall t \ge 0\right) \le \frac{\exp\{c\lambda^2\deg(o)\}}{T} + T \sum_{\gamma \in \Gamma_A \cup \Gamma_{A,*}}(2\lambda)^{|\gamma|} \qquad \text{for all }T > 0.$$ 
\end{lemma}
\begin{proof}
 Let~$S$ denote the star graph with vertex set~$\{o\}\cup\{x:x\sim o\}$ and edge set~$\{\{o,x\}:x\sim o\}$ (we will also denote the vertex set of this  graph by~$S$). We assume given a graphical construction for the contact process~$(\xi_t)$ with rate~$\lambda$ on~$G$; using this same graphical construction, we define~$(\eta_t)$ as the contact process on~$S$ with~$\eta_0 = \{o\}$.

Fix~$T > 0$. Let~$\tau = \inf\{t:\eta_t = \varnothing\}$ and define the event~$ E_o := \{\tau \ge T\}$. For each finite~$\gamma = (\gamma(0),\ldots, \gamma(k)) \in \Gamma_\infty$, let~$ E_{T,\gamma}$ denote the event that there exist $t < T$ and an infection path starting at~$(o,t)$ and having ordered trace~$\gamma$. Finally, define~$\tau'$ as the first time when either a vertex of~$A$ becomes infected, or an infection path~$g:[0,\tau'] \to V$ can be formed with~$g(0)=o$ and so that some vertex~$v \notin S$ is in the ordered trace of~$g$ twice.

\begin{claim} We have that
\begin{equation}
\{\tau' < \infty\} \subset E_o \cup \bigcup_{\gamma \in \Gamma_A \cup \Gamma_{A,*}}E_{T,\gamma}.\label{eq:union_ev}
\end{equation} \label{cl:event_inclusion}
\end{claim}
\begin{proof}[Proof of Claim~\ref{cl:event_inclusion}]
Assume that~$\tau' < \infty$. Then, we can take an infection path~$g:[0,\tau'] \to V$ with~$g(0) = o$ and so that either~$g(\tau') \in A$ or the ordered trace of~$g$ contains some vertex~$v \notin S$ more than once. We consider three cases:
\begin{itemize}
\item If~$\tau' \ge T$ and  during the whole time interval~$[0,T]$,~$g$ only occupies vertices of~$S$, and only traverses edges of~$S$, then~$E_o$ occurs.
\item If~$\tau' < T$ and during the whole time interval~$[0,\tau']$,~$g$ only occupies vertices of~$S$, and only traverses edges of~$S$   (which can only happen if~$g(\tau') \in S \cap A$), then the event~$E_{T,\gamma}$ occurs for~$\gamma = (o,g(\tau'))$.
\item If neither of the previous two situations holds, then we let~$s$ be the first time at which~$g$ traverses an edge that is not in~$S$; note that~$s \le T$,~$g(s-)$ is a vertex of~$S$, and~$g(s)$ may or may not be a vertex of~$S$. Then,~$E_{T,\gamma}$ occurs  for the vertex path~$\gamma$ defined by setting~$\gamma(0)=o$,~$\gamma(1) = g(s-)$,~$\gamma(2) = g(s)$, and the rest of~$\gamma$ given by the subsequent  vertices visited by~$g$ in order, stopping when either there is a repetition or~$A$ is reached.
\end{itemize}
\end{proof}

We now complete the proof of the lemma by using the claim and bounding the probabilities of the events on the right-hand side of~\eqref{eq:union_ev}. It is known that there exists~$c > 0$ such that~$\mathbb{E}[\tau] \le \exp\{c\lambda^2\mathrm{deg}(o)\}$ (see Theorem~1.4 in~\cite{HD20} and the observation that follows it). 
  Using this and Markov's inequality,
$$\mathbb{P}(E_o) \le \frac{\exp\{c \lambda^2\deg(o)\}}{T}.$$
Next, fix~$\gamma = (\gamma(0),\ldots, \gamma(k)) \in \Gamma_A \cup \Gamma_{A,*}$. Let us first observe that, for any~$t$, the probability that there is an infection path starting at~$(\gamma(1),t)$ and from there visiting the vertices~$(\gamma(2),\ldots,\gamma(k))$ in order is smaller than~$(2\lambda)^{k-1}$, by Lemma~\ref{lem:super}. Hence, letting~$\mathcal{X}$ denote the set of times~$t \le T$ at which there is a transmission arrow from~$(o,t)$ to~$(\gamma(1),t)$, a union bound gives~$\mathbb{P}(E_{T,\gamma}|\mathcal{X}) \le |\mathcal{X}|\cdot(2\lambda)^{k-1}$. Taking expectations on both sides of this inequality, we obtain:
$$\mathbb{P}(E_{T,\gamma}) \le (2\lambda)^{k-1}\cdot \mathbb{E}[|\mathcal{X}|] = T\cdot \lambda^k.$$
Hence, by a union bound over all~$\gamma$, the probability that a vertex of~$A$ ever becomes infected, or that a vertex outside~$S$ appears more than once in the ordered trace of an infection path started from~$(o,0)$, is at most
$$\frac{\exp\{c\lambda^2\deg(o)\}}{T} + T \sum_{\gamma \in \Gamma_A \cup \Gamma_{A,*}}(2\lambda)^{|\gamma|}.$$
If none of these things happen, then~$(\xi_t)$ is thing outside~$S$. The conclusion now follows from Lemma~\ref{lem:Afin}.
\end{proof}

{We now come back to the hyperbolic setup.} 
 Given~$u = (x_u,h_u),\;v=(x_v,h_v) \in \mathbb{H}$ with~$|x_u-x_v| \le \exp\left\{(h_u+h_v)/2\right\}$, let~$\mathbf{G}^{u,v}$ denote the graph obtained from~$\mathbf{G}_\infty$ by artificially including vertices at~$u$ and~$v$. We root this graph at~$u$. We define \begin{equation}\label{eq:def_of_hstar}h_\star := H\left(\frac{1}{\lambda^2}\right)\end{equation}
and $$A :=\{w = (x_w,h_w)\in\mathbf{G}^{u,v}:\;h_w\ge h_\star \},$$ and the sets of vertex paths~$\Gamma_A$ and~$\Gamma_{A,*}$ as in~\eqref{eq:def_gammas}. We then have:
\begin{lemma}\label{lem:bound_gp}
There exists~$\varepsilon_0> 0$ such that for any~$\delta > 0$ and for~$\lambda$ small enough (depending on~$\delta$), the following holds. Abbreviate
\begin{equation}\label{eq:def_of_hpp}
h'':= H\left(\frac{\delta}{\lambda^2}\log\left(\frac{1}{\lambda}\right)\right).
\end{equation}
 If~$u$ has height~$h_u \le h''$ and~$v$ has height~$h_v \le h_\star$, then
$$\mathbb{E}\left[\sum_{\gamma \in \Gamma_A \cup \Gamma_{A,*}}(2\lambda)^{|\gamma|} \right] < \lambda^{\varepsilon_0}.$$
\end{lemma}
We will give the proof of this lemma later; for now, we state and prove:

\begin{proposition}
\label{prop:guv} There exist~$\delta, \varepsilon > 0$ such that the following holds for~$\lambda$ small enough. Let~$u=(x_u,h_u),~v=(x_v,h_v)$ be as above, and further assume that 
\begin{equation}\label{eq:assume_hu}h_u \le h'',\qquad h_v \le h_\star.\end{equation}
Let~$(\xi_t)$ denote the contact process with parameter~$\lambda$ on~$\mathbf{G}^{u,v}$ and~$\xi_0 = \{u\}$. Then,
$$\mathbb{P}\left(\xi_t \neq \varnothing \;\forall t \ge 0\right) \le \lambda^\varepsilon.$$
\end{proposition}
\begin{proof} We let~$\delta = \frac{\varepsilon_0}{8c}$, where~$\varepsilon_0$ is the constant of Lemma~\ref{lem:bound_gp}, and~$c$ is the constant of Lemma~\ref{lem:star_lem}. Also let~$T = \lambda^{-\varepsilon_0/2}$. Then, by Lemma~\ref{lem:star_lem},
\begin{align*}
&\mathbb{P}(\xi_t \neq \varnothing \;\forall t \ge 0 \mid \mathbf G^{u,v})  \\
&\le \mathds{1}\left\{\deg(u) > \frac{2\delta}{\lambda^2}\log\left(\frac{1}{\lambda}\right)\right\} +\frac{\exp\{c\lambda^2\cdot \frac{2\delta}{\lambda^2}\log\left(\frac{1}{\lambda}\right)\}}{T} + T \sum_{\gamma \in \Gamma_A \cup \Gamma_{A,*}}(2\lambda)^{|\gamma|} \\[.2cm]
&=  \mathds{1}\left\{\deg(u) > \frac{2\delta}{\lambda^2}\log\left(\frac{1}{\lambda}\right)\right\} + \lambda^{\varepsilon_0/4} + \lambda^{-\varepsilon_0/2}\cdot \sum_{\gamma \in \Gamma_A \cup \Gamma_{A,*}}(2\lambda)^{|\gamma|}.
\end{align*}
Taking expectations and using Lemma~\ref{lem:bound_gp} then gives
$$\mathbb{P}(\xi_t \neq \varnothing \;\forall t \ge 0) \le \mathbb{P}\left(\deg(u) > \frac{2\delta}{\lambda^2}\log\left(\frac{1}{\lambda}\right)\right) +\lambda^{\varepsilon_0/4}+\lambda^{\varepsilon_0/2}.$$
Note that $\deg(u)-1\sim \text{Poisson}(D(h_u))$ and by~\eqref{eq:assume_hu} we have
$$D(h_u) \le D\left(H\left(\frac{\delta}{\lambda^2}\log\left(\frac{1}{\lambda}\right)\right)\right) = \frac{\delta}{\lambda^2}\log\left(\frac{1}{\lambda}\right).$$ 
Using a Chernoff bound, it is easy to see that there exists~$\bar{c} > 0$ such that
\begin{align*}
&\mathbb{P}\left(\deg(u) > 2\frac{\delta}{\lambda^2}\log\left(\frac{1}{\lambda}\right)\right)\leq \exp\left\{-\bar{c}\cdot \frac{\delta}{\lambda^2}\log\left(\frac{1}{\lambda}\right) \right\}\ll \lambda, 
\end{align*}
if~$\lambda$ is small. We then have, for~$\lambda$ small,
$$\mathbb{P}(\xi_t \neq \varnothing \;\forall t \ge 0) \le \lambda+ \lambda^{\varepsilon_0/4}+\lambda^{\varepsilon_0/2},$$
so the result follows by taking~$\varepsilon = \varepsilon_0/5$.
\end{proof}

\begin{proof}[Proof of Lemma~\ref{lem:bound_gp}]
We fix~$\varepsilon_0 > 0$, whose value will be chosen later, let~$\delta > 0$ be arbitrary, and assume~$u = (x_u,h_u)$ has~$h_u \le h''$, with~$h''$ defined as in~\eqref{eq:def_of_hpp}. Recall that~$\Gamma_A^k = \{\gamma \in \Gamma_A:\;|\gamma| = k\}$ for~$k \ge 1$ and~$\Gamma_{A,*}^k = \{\gamma \in \Gamma_{A,*}:\;|\gamma| = k\}$ for~$k \ge 3$. We further let~$\hat{\Gamma}_A^k$ be the set of vertex paths in~$\Gamma_A^k$ that do not visit~$v$, and similarly define~$\hat{\Gamma}_{A,*}^k$.

 We bound, for~$k \ge 1$, {using the multivariate Mecke formula (see~\cite[Theorem 4.4]{Last-Penrose}) and Lemma~\ref{lem:super}}, 
\begin{align*}
(2\lambda)^k\cdot \mathbb{E}[|\hat\Gamma_A^k|] &\le (2\lambda)^k\int_{h^{(1)}<h_\star}\cdots \int_{h^{(k-1)} < h_\star} \int_{h^{(k)} \ge h_\star}\mathrm{d}h^{(k)}\cdots \mathrm{d}h^{(1)}\\ &\quad\exp\left\{\frac{h_u}{2}+(1-\alpha)(h^{(1)}+\cdots +h^{(k-1)})+ \left(\frac12-\alpha\right)h^{(k)} \right\}\\
&= (2\lambda)^k \cdot C^k\cdot \exp\left\{\frac{h_u}{2} +\left((1-\alpha)(k-1) + \frac12 - \alpha\right)h_\star\right\},
\end{align*}
(Recall that the value of $C=C(\alpha)$ may change from line to line, but it will never depend on $\lambda$).
Recalling that~$h_u \le h''= H\left(\frac{\delta}{\lambda^2}\log\left(\frac{1}{\lambda}\right)\right)$,~$h_\star = H\left(\lambda^{-2}\right)$ and~\eqref{eq:d_and_h}, we see that the above is smaller than
\begin{align*}&(2\lambda)^k C^k \frac{1}{\lambda^2}\log\left(\frac{1}{\lambda}\right) \left(\frac{1}{\lambda}\right)^{4\left((1-\alpha)(k-1) + \frac12-\alpha\right)} = C^k \log \left(\frac{1}{\lambda}\right) \lambda^{(4\alpha -3)k}. \end{align*}
Hence,
\begin{equation}\label{eq:prescr1}
\sum_{k=1}^\infty (2\lambda)^k \cdot \mathbb{E}\left[|\hat\Gamma^k_A|\right] \le C\cdot \log\left(\frac{1}{\lambda}\right)\cdot \lambda^{4\alpha - 3} < \lambda^{\tilde \varepsilon}
\end{equation}
for some~$\tilde \varepsilon > 0$ and~$\lambda$ small enough, since~$4\alpha - 3 > 0$.

Next, for~$k \ge 3$, {again by the multivariate Mecke formula}, 
\begin{align*}
(2\lambda)^k\cdot \mathbb{E}[|\hat\Gamma_{A,*}^k|] &\le (2\lambda)^k \cdot k\int_{h^{(1)}<h_\star}\cdots \int_{h^{(k-1)}<h_\star}dh^{(k-1)}\cdots dh^{(1)}\\
&\quad\exp\left\{\frac{h_u}{2} + (1-\alpha)(h^{(1)}+\cdots + h^{(k-2)}) + \left(\frac12-\alpha \right)h^{(k-1)} \right\}\\[.2cm]
&\le (2\lambda)^k\cdot  k\cdot C^k \cdot \frac{1}{\lambda^2}\log\left(\frac{1}{\lambda}\right) \cdot \left(\frac{1}{\lambda}\right)^{4(1-\alpha)(k-2)}\\[.2cm]
&= k \cdot C^k \cdot \log\left(\frac{1}{\lambda}\right) \cdot \lambda^{(4\alpha -3)k -8\alpha + 6}.
\end{align*}
Then,
\begin{equation}\label{eq:prescr2}
\sum_{k=3}^\infty(2 \lambda)^k\cdot \mathbb{E}\left[|\hat\Gamma_{A,*}^k|\right] \le C \cdot \log\left(\frac{1}{\lambda}\right)\cdot \lambda^{(4\alpha -3)\cdot 3 - 8\alpha +6} < \lambda^{\tilde \varepsilon}
\end{equation}
for some~$\tilde \varepsilon > 0$ and~$\lambda$ small enough, since the exponent of~$\lambda$ in the middle term is~$4\alpha - 3 > 0$.

The bounds carried out above, yielding~\eqref{eq:prescr1} and~\eqref{eq:prescr2}, can be repeated for the sets of vertex paths~$\Gamma_A^k \backslash \hat \Gamma_A^k$ and~$\Gamma_{A,*}^k \backslash \hat \Gamma_{A,*}^k$, with no significant differences, except that one of the integrals involved in each of the bounds is  suppressed to account for a visit to~$v$. We omit the details for brevity. The result now follows by taking~$\varepsilon_0 < \tilde \varepsilon$ and~$\lambda$ small.
\end{proof}
\subsection*{Bounds on infection paths through low vertices}
We let~$(\mathbf{G}_\infty,o)$ be the random hyperbolic graph on~$\mathbb{H}$ with uniformly chosen root, and~$(\xi_t)_{t \ge 0}$ the contact process with rate~$\lambda$ on this graph with~$\xi_0 = \{o\}$.

Our next goal is  to prove:
\begin{proposition}\label{prop:bdwl}
There exists~$\varepsilon_1 > 0$ and~$\sigma > 0$ such that the following holds for~$\lambda$ small enough. Abbreviate
\begin{equation}
\label{eq:hp} h':= H\left(\frac{1}{\lambda^{2-\sigma}}\right).
\end{equation}
 Let~$\bar{E}$ be the event that: for every infection path~$g$ which starts at~$o$ at time zero, and from there jumps to a vertex~$v = (x_v,h_v)$ with~$h_v\le h'$, we have that~$g$ is finite and never visits a vertex with height above~$h'$. Then,
$$\mathbb{P}(\bar{E}) \ge 1 - \lambda^{4\alpha - 1+\varepsilon_1}. $$
\end{proposition}
Before proving this result, we need to give some definitions, and state and prove a lemma. We continue abbreviating~$h_\star = H\left(\frac{1}{\lambda^2}\right)$. We leave~$\sigma \in (0,1)$ fixed for now, with~$h'$ as in~\eqref{eq:hp} and we define the random vertex set
\begin{equation*}\mathcal{A} = \mathcal{A}_\sigma:=\{v = (x_v,h_v) \in \mathbf{G}_\infty:\;v \neq o,\;   h_v \ge h'\}.\end{equation*}
Next, define~$\Gamma_\mathcal{A} = \Gamma_\mathcal{A}(\mathbf{G}_\infty,o)$ and~$\Gamma_{\mathcal{A},*}= \Gamma_{\mathcal{A},*}(\mathbf{G}_\infty,o)$ as in~\eqref{eq:def_gammas}; also let
$$\tilde \Gamma_{\mathcal{A}} := \Gamma_{\mathcal{A}}\backslash \Gamma_{\mathcal{A}}^1 = \cup_{k \ge 2} \;\Gamma_{\mathcal{A}}^k$$
and
$$\Gamma_0 = \{(o,u,o,v):u,v \sim o\}.$$
\begin{lemma}\label{lem:inclusion} Assume~$\sigma \in (0,1)$.
If no infection path~$g$ with~$g(0) = o$ has~$\gamma_g\in\Gamma_0 \cup \tilde\Gamma_\mathcal{A} \cup \Gamma_{\mathcal{A},*}$, then the event~$\bar{E}$ of Proposition~\ref{prop:bdwl} occurs: any infection path~$g$ with~$g(0) = o$ and~$\gamma_g(1)\notin \mathcal{A}$ is finite and never enters~$\mathcal{A}$.
\end{lemma}
\begin{proof}
Assume that the realization~$H$ of the graphical construction of the contact process is such that  no infection path started at~$o$ at time zero  has ordered trace in~$\Gamma_0 \cup \tilde\Gamma_\mathcal{A} \cup \Gamma_{\mathcal{A},*}$. Then, it is readily seen that, for any infection path~$g$ (starting from time zero),
\begin{equation}\label{eq:first_m_prop}
\text{if }g(0)=o,\;\gamma_g(1)\notin\mathcal{A},\quad  \text{then } g\text{ does not intersect }\mathcal{A},
\end{equation}
and also
\begin{equation}\label{eq:sec_m_prop}
\text{if }g(0)=o,\;\gamma_g(1)\notin\mathcal{A},\quad  \text{then } \text{no }u\neq o \text{ appears in }\gamma_g\text{ more than once}.
\end{equation}
(indeed, if an infection path~$g:[0,t]\to \mathbf{G}_\infty$ with~$g(0) = o$ and~$\gamma_g(1)\notin \mathcal{A}$ violated either property, we could obtain~$s \le t$ so that the restriction~$\tilde g$ of~$g$ to~$[0,s]$ would have~$\gamma_{\tilde{g}} \in \Gamma_0 \cup \tilde \Gamma_{\mathcal{A}} \cup \Gamma_{\mathcal{A},*}$).

Now, let~$H'$ denote the graphical construction obtained by removing from~$H$ all Poisson processes associated to vertices of~$\mathcal{A}$, and edges that intersect~$\mathcal{A}$. Then,~\eqref{eq:first_m_prop} implies that the set of~$H$-infection paths~$g$ with~$g(0)=o,\;\gamma_g(1)\notin\mathcal{A}$ is equal to the set of~$H'$-infection paths~$g$ with~$g(0) = o$. Moreover,~\eqref{eq:sec_m_prop} implies that the contact process~$(\xi'_t)_{t \ge 0}$ obtained from~$H'$ and~$\xi'_0 =\{o\}$ is thin outside~$o$, so by Lemma~\ref{lem:Afin}, this process dies out. In particular, any~$H'$-infection path~$g$ with~$g(0) = o$ is finite.
\end{proof}
\begin{proof}[Proof of Proposition~\ref{prop:bdwl}.]
Recalling that~$h_o$ denotes the height of the root~$o$, we start by bounding
\begin{align}\nonumber\mathbb{P}(\bar{E}^c) &\le \mathbb{P}(h_o > h') + \mathbb{P}(\bar{E}^c \cap \{h_o \le h'\})\\
&\le \mathbb{P}(h_o > h') + \mathbb{E}\left[\mathds{1}\{h_o \le h'\}\cdot \sum_{\gamma \in \Gamma_0 \cup \tilde \Gamma_{\mathcal{A}}\cup \Gamma_\mathcal{A,*}}(2\lambda)^{|\gamma|} \right],\label{eq:assemble}
\end{align}
where the second inequality follows from Lemmas~\ref{lem:super} and~\ref{lem:inclusion}. We will bound the terms on the right-hand side separately. We start with
\begin{equation}\begin{split}
\mathbb{P}(h_o > h') &=\int_{h'}^\infty \alpha e^{-\alpha h_o}\;\mathrm{d}h_o\stackrel{\eqref{eq:d_and_h},\eqref{eq:hp}}{\le}  C \left(\dfrac{1}{\lambda}\right)^{-2(2-\sigma)\alpha} < \lambda^{4\alpha - 1+\varepsilon_1}\end{split}\label{eq:bound_prob}
\end{equation}
for~$\sigma> 0$ and~$\varepsilon_1>0$ small enough, and then~$\lambda$ small enough.

Next, we bound {(again using the multivariate Mecke formula):}
\begin{align*}
&(2\lambda)^3\cdot \mathbb{E}\left[|\Gamma_0| \cdot \mathds{1}\{h_o \le h'\}\right]\\
&\le (2\lambda)^3\int_0^{h'}\alpha \int_0^\infty \int_0^\infty \mathrm{d}h_o\mathrm{d}h^{(1)}\mathrm{d}h^{(2)} \exp\left\{(1-\alpha)h_o +\left(\frac{1}{2}-\alpha\right)(h^{(1)}+h^{(2)})\right\}\\
&\le \lambda^3 \cdot C \cdot \exp\left\{(1-\alpha)h'\right\} \le C\cdot \lambda^{3-2(2-\sigma)(1-\alpha)}.
\end{align*}
Now, since
$$3-2(2-\sigma)(1-\alpha) > 3-2\cdot 2\cdot(1-\alpha) = 4\alpha - 1, $$
we obtain
\begin{equation} \label{eq:bound_gamma_zero}
(2\lambda)^3 \cdot \mathbb{E}\left[|\Gamma_0|\cdot \mathds{1}\{h_o \le h'\}\right] <\lambda^{4\alpha - 1+\varepsilon_1}
\end{equation}
for some~$\varepsilon_1 > 0$ and~$\lambda$ small enough.

We now bound, for~$k \ge 2$, {one more time using the multivariate Mecke formula,}
\begin{align*}
(2\lambda)^k\cdot \mathbb{E}\left[|\Gamma_{\mathcal{A}}^k|\right] &= (2\lambda)^k \int_0^\infty \int_{h^{(1)}<h'}\cdots \int_{h^{(k-1)}<h'}\int_{h^{(k)}\ge h'}\mathrm{d}h^{(k)}\cdots \mathrm{d}h^{(1)}\mathrm{d}h_o\\
&\;\;\alpha\exp\left\{(1-\alpha)(h^{(1)}+\cdots + h^{(k-1)}) + \left(\frac12-\alpha \right)(h_o+h^{(k)})\right\}\\[.2cm]
&\le \lambda^k\cdot C^{k+1}\cdot \exp\left\{\left((1-\alpha)(k-1)+\frac12 - \alpha\right) h'\right\} \\[.2cm]
&\le C^{k+1} \cdot \lambda^{k-2(2-\sigma)[(1-\alpha)(k-1) + \frac12 - \alpha]}.
\end{align*}
Thus, if~$\sigma$ is small,
\begin{equation}\label{eq:bound_gamma_A2}
\sum_{k=2}^\infty (2\lambda)^k\cdot \mathbb{E}\left[|\Gamma^k_{\mathcal{A}}|\right] \le C \cdot \lambda^{2-2(2-\sigma)[(1-\alpha)(2-1)+\frac12-\alpha]} < \lambda^{4\alpha - 1 + \varepsilon_1},
\end{equation}
{where for the last inequality we assumed that~$\varepsilon_1 > 0$ is small enough} depending on~$\alpha$, and~$\sigma=\sigma(\varepsilon)$ is small enough. Indeed, this can be accomplished, since if we had $\sigma=0$, then the exponent of~$\lambda$ in the middle term would be~$8\alpha - 4$, which is strictly larger than~$4\alpha - 1$ when~$\alpha > \frac34$; by continuity, this strict inequality still holds for small~$\sigma > 0$.

The last term we have to treat is, for~$k \ge 3$ {(again using the multivariate Mecke formula)}
\begin{align*}
&(2\lambda)^k\cdot \mathbb{E}\left[|\Gamma_{\mathcal{A},*}^k|\right] \le (2\lambda)^k\cdot k \int_0^\infty \int_{h^{(1)}<h'}\cdots \int_{h^{(k-1)}<h'}\mathrm{d}h^{(k-1)}\cdots \mathrm{d}h^{(1)}\mathrm{d}h_o \\
&\hspace{3cm} \alpha \exp\left\{\left(\frac12-\alpha\right)(h_o+h^{(k-1)})+(1-\alpha)(h^{(1)}+\cdots + h^{(k-2)})\right\}\\[.2cm]
&\leq \lambda^k \cdot C^k \cdot \exp\left\{ (k-2)(1-\alpha)h'\right\} \le C^k \cdot \lambda^{k-2(2-\sigma)(1-\alpha)(k-2)}.
\end{align*}
Then,
\begin{equation}\label{eq:bound_rep}
\sum_{k=3}^\infty (2\lambda)^k\cdot \mathbb{E}\left[|\Gamma_{\mathcal{A},*}^k|\right] \le C \cdot \lambda^{3-2(2-\sigma)(1-\alpha)(3-2)}< \lambda^{4\alpha -1+\varepsilon_1}
\end{equation}
for small~$\varepsilon_1 > 0$: if we had~$\sigma = 0$, then the exponent of~$\lambda$ in the middle term would be precisely~$4\alpha- 1$, and moreover this exponent is increasing in~$\sigma$.

The proof is now completed by using the bounds~\eqref{eq:bound_prob},~\eqref{eq:bound_gamma_zero},~\eqref{eq:bound_gamma_A2} and~\eqref{eq:bound_rep} back in~\eqref{eq:assemble}.
\end{proof}
{We are now prepared to finish the proof of Proposition~\ref{upper341}.}
\par \noindent 
\textbf{Proof of Proposition~\ref{upper341}.}
We recall the definition of~$h_\star$,~$h''$ and~$h'$ in~\eqref{eq:def_of_hstar},~\eqref{eq:def_of_hpp} and~\eqref{eq:hp}. 
We now give several additional definitions. We let
$$\mathcal{N} = \{v=(x_v,h_v) \in \mathbf{G}_\infty:\; v \sim o,\;  h_v \ge h'\},\quad N = |\mathcal{N}|.$$
On the event~$\{N = 1\}$, we define~$\hat{u} = (x_{\hat{u}},h_{\hat{u}})$ as the unique element of~$\mathcal{N}$.
Next, let~$M$ denote the number of transmission arrows that appear from~$o$ to vertices of~$\mathcal{N}$ before the first recovery mark at~$o$. On the event~$\{N= 1,\;M \ge 1\}$, define~$\tau$ as the first time a transmission arrow occurs from~$o$ to~$\hat{u}$. Further define, on~$\{N=1,\;M\ge 1\}$, the process~$(\eta_{t})_{t\ge \tau}$ as the contact process on~$\mathbf{G}_\infty$ started from time~$\tau$, with a single infection at~$\hat{u}$; this process is defined with the same graphical construction as that of the original process on~$\mathbf{G}_\infty$. In other terms, recalling the notation from Section~\ref{ss:contact}, we set
$$\eta_t(v) = 1\left\{(\hat{u},\tau) \rightsquigarrow (v,t)\right\},\quad v \in \mathbf{G}_\infty,\; t \ge \tau. $$
Lastly, we define the event
$$\widehat{E}:= \{N=1,\;M = 1,\;\eta_t \neq \varnothing \text{ for all }t \ge \tau\}.$$

Recall the definition of the event~$\bar{E}$ in Proposition~\ref{prop:bdwl}. We now claim that, if neither of the four events
\begin{equation}
(\bar{E})^c,\quad  \{N \ge 2,\;M \ge 1\},\quad \{N =1,\;M \ge 2\}, \quad \widehat{E} \label{eq:four_events}
\end{equation}
occurs, then~$\xi_t = \varnothing$ for some~$t$. To prove this, we first observe that, by the definition of~$\bar{E}$, on the event~$\bar{E}\cap \{M=0\}$ we have that every infection path started at~$o$ at time zero is finite, and hence~$(\xi_t)$ dies out. Having this in mind, if neither of the four events in~\eqref{eq:four_events} occur, the only remaining situation in which we need to rule out the survival of~$(\xi_t)$ is when~$N=M=1$ and~$(\eta_t)_{t \ge \tau}$ dies out: this is the area painted blue in Figure~\ref{fig:dg}. In that case we can argue as follows: given an infection path~$g$ started at~$o$ at time zero, if we have~$\gamma_{g}(1) \neq \hat{u}$ then~$g$ is finite (because~$\bar{E}$ occurs), and if~$\gamma_g(1) = \hat{u}$, then the jump from~$o$ to~$\hat{u}$ must be through the only transmission arrow from~$o$ to~$\hat{u}$ before the first recovery at~$o$; then, the rest of~$g$ is an infection path available to~$(\eta_t)$, so it is finite since~$(\eta_t)$ dies out.

\begin{figure}[htb]
\begin{center}
\setlength\fboxsep{0pt}
\setlength\fboxrule{0.5pt}
\fbox{\includegraphics[width =.8 \textwidth]{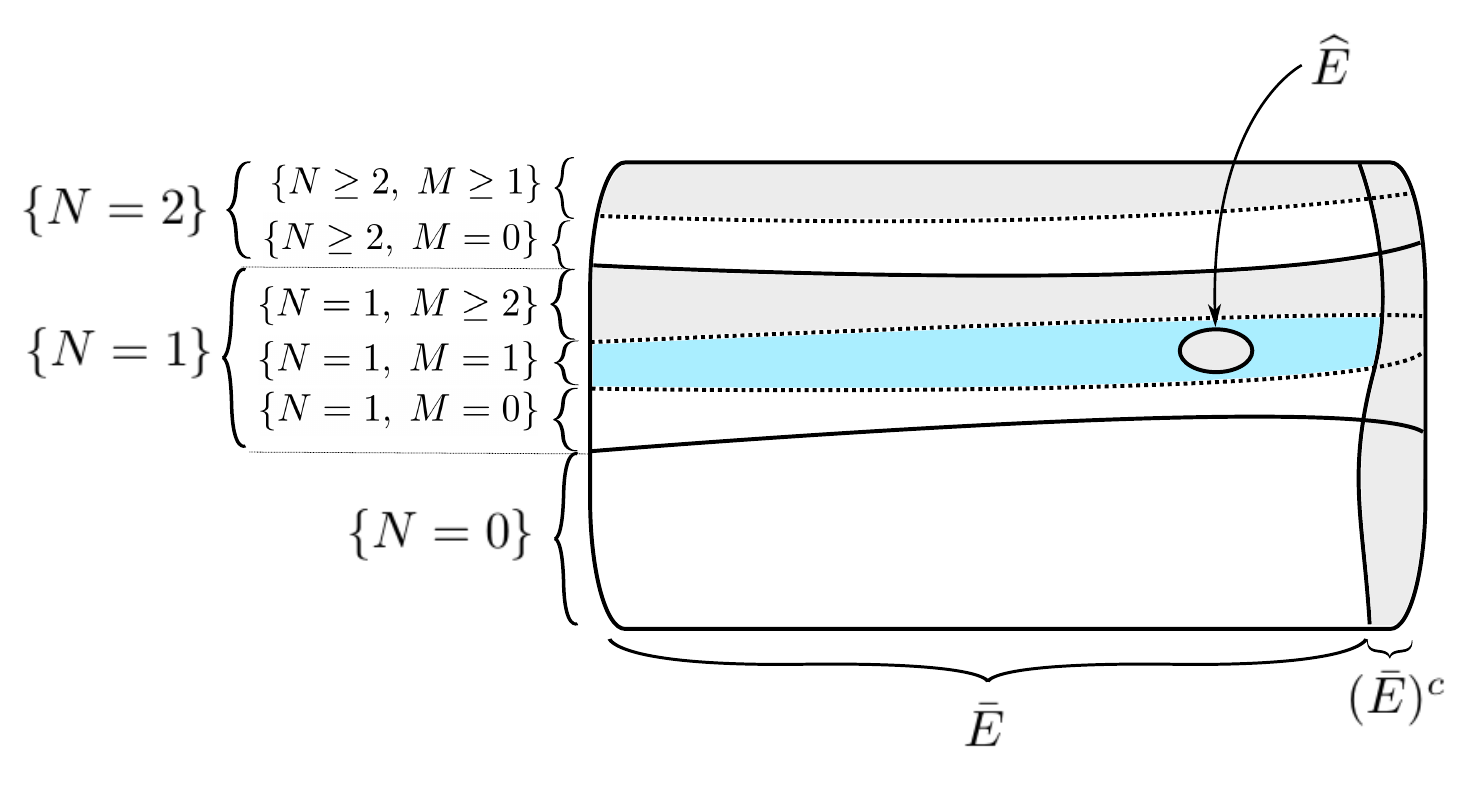}}
\end{center}
\vspace{-.4cm}\caption{The four bad events defined in~\eqref{eq:four_events} are painted grey. In the regions painted white and blue, the contact process dies out.}
\label{fig:dg}
\end{figure}

Hence, the proof of the upper bound will be complete once we show that the four events in~\eqref{eq:four_events} have probability smaller than~$C\frac{\lambda^{4\alpha-1}}{\log(1/\lambda)^{2\alpha-1}}$, for some~$C>0$. For~$(\bar{E})^c$, this is already given by Proposition~\ref{prop:bdwl}. We proceed to bound the other ones in order.\\

$\bullet$ Probability of~$\{N \ge 2,\; M \ge 1\}$. We bound
\begin{equation}\label{eq:bound_e_N}\mathbb{P}(N \ge 2,\;M \ge 1) = \mathbb{E}\left[\frac{\lambda N}{1+\lambda N}\cdot \mathds{1}\{N \ge 2\}\right] \le \lambda \cdot \mathbb{E}[N \cdot \mathds{1}\{N \ge 2\}].\end{equation}
The law of~$N$ conditioned on~$h_o$ is Poisson with parameter
\begin{equation}e^{\frac{h_o}{2}}\int_{h'}^\infty \exp\left\{ \left(\frac12-\alpha\right)h\right\}\;\mathrm{d}h \le C \cdot e^{\frac{h_o}{2}}\cdot \lambda^{(2-\sigma)\left(2\alpha - 1\right)},\label{eq:poi_par}\end{equation}
so we can bound
\begin{align*}&\mathbb{E}[\mathbb{E}[N\cdot \mathds{1}\{N \ge 2\}\mid h_o]\cdot \mathds{1}\{h_o \ge H(1/\lambda)\}] \le \mathbb{E}[\mathbb{E}[N\mid h_o]\cdot \mathds{1}\{h_o \ge H(1/\lambda)\}]\\[.2cm]
&\le C \cdot \int_{H(1/\lambda)}^\infty e^{-\alpha h_o}\cdot \left(e^{\frac{h_o}{2}} \cdot \lambda^{(2-\sigma)\left(2\alpha-1\right)}\right)\mathrm{d}h_o \le C \cdot \lambda^{(3-\sigma)(2\alpha - 1)}.
\end{align*}
Next, when~$h_o < H(1/\lambda)$ the expression on the right-hand side of~\eqref{eq:poi_par} is smaller than~$C \cdot \lambda^{(2-\sigma)(2\alpha - 1) - 1} \ll 1$ if~$\sigma$ is small (and~$\lambda$ is small), since~$\alpha > \frac34$. We then use the bound, for~$Z \sim \text{Poisson}(\beta)$ and~$\beta$ small,
$$\mathbb{E}[Z\cdot \mathds{1}\{Z \ge 2\}] = \mathbb{E}[Z] - \mathbb{E}[Z\cdot \mathds{1}\{Z = 1\}] = \beta - \beta e^{-\beta} \le \beta^2$$ 
to obtain
\begin{align*}
&\mathbb{E}[\mathbb{E}[N\cdot \mathds{1}\{N \ge 2\}\mid h_o]\cdot \mathds{1}\{h_o < H(1/\lambda)\}]\\[.2cm]
&\le C \cdot \int_0^{H(1/\lambda)} e^{-\alpha h_o}\cdot \left(e^{\frac{h_o}{2}}\cdot \lambda^{(2-\sigma)(2\alpha-1)}\right)^2\mathrm{d}h_o \\[.2cm]&= C \cdot \lambda^{(2-\sigma)(4\alpha - 2)-2(1-\alpha)}.
\end{align*}

Now the expression on the right-hand side of~\eqref{eq:bound_e_N} is smaller than
$$C \cdot\left(\lambda^{(3-\sigma)(2\alpha - 1)+1}+\lambda^{(2-\sigma)(4\alpha - 2)-2(1-\alpha)+1} \right).$$
If we had~$\sigma = 0$, the exponents of~$\lambda$ inside the parentheses would be~$6\alpha -2$ and~$10\alpha - 5$, both of which are larger than~$4\alpha -1$ when~$\alpha > \frac34$. This shows that
$$\mathbb{P}(N \ge 2,\; M \ge 1)\le \lambda^{4\alpha - 1 + \varepsilon'}$$
for some~$\varepsilon' > 0$, if~$\sigma$ is small enough (and~$\lambda$ is small).\\

$\bullet$ Probability of~$\{N = 1,\; M \ge 2\}$. This is easier to handle. We {first have (by the multivariate Mecke formula)}
\begin{equation}\begin{split}\label{eq:bound_NN}
\mathbb{P}(N = 1) &\le  \int_0^\infty \int_{h'}^\infty \exp\left\{\left(\frac12-\alpha\right)(h_o+h)\right\}\mathrm{d}h\mathrm{d}h_o \\[.2cm]&\le C \cdot \exp\left\{\left(\frac12-\alpha\right)h' \right\} = C \cdot \lambda^{(2-\sigma)(2\alpha-1)}.\end{split}
\end{equation}
Then, we bound
\begin{equation*}
\mathbb{P}(N =1,\;M \ge 2) = \left(\frac{\lambda}{1+\lambda}\right)^2\cdot \mathbb{P}(N=1)\stackrel{\eqref{eq:bound_NN}}{\le}C \cdot \lambda^{2+2(2-\sigma)(\alpha - \frac12)} < \lambda^{4\alpha - \frac12}
\end{equation*}
if~$\sigma$ is small enough, and then~$\lambda$ is small enough.\\

$\bullet$ Probability of~$\widehat{E}$. We start with
\begin{align*}\mathbb{P}(\widehat{E})& \le\mathbb{P}(h_o > h_\star)+ \mathbb{P}(N = 1,\;M\ge 1,\; h_{\hat{u}}> h'') \\&\quad+ \mathbb{P}\left(N= 1,\; M \ge 1,\;h_o \le h_\star,\;h_{\hat{u}} \le h'',\;(\eta_t)_{t \ge \tau} \text{ survives}\right).\end{align*}
The first two terms can be handled with some more calculations of integrals:
$$\mathbb{P}(h_o > h_\star) \le C \cdot\exp\{-\alpha h_\star\} < C \cdot \lambda^{4\alpha},$$
as in~\eqref{eq:bound_prob}, and
\begin{align*}&\mathbb{P}(N = 1,\;M\ge 1,\; h_{\hat{u}}> h'')  \\[.2cm]
&\le \frac{\lambda}{1+\lambda}\cdot\int_0^\infty \int_{h''}^\infty \alpha \exp\left\{-\left(\alpha+\frac12\right) (h_o+h)\right\}\;\mathrm{d}h\mathrm{d}h_o \le C \cdot \frac{\lambda^{4\alpha - 1}}{\log(1/\lambda)^{2\alpha - 1}}\end{align*}
(this is the only term in the proof whose bound is at the sharp value). Next,
\begin{align}\nonumber
& \mathbb{P}\left(N= 1,\; M \ge 1,\;h_o \le h_\star,\;h_{\hat{u}} \le h'',\;(\eta_t)_{t \ge \tau} \text{ survives}\right)\\
&\le \frac{\lambda}{1+\lambda}\cdot \mathbb{E}\left[\mathbb{P}\left((\eta_t)_{t \ge \tau}\text{ survives}\mid o,\hat{u}\right)\cdot \mathds{1}\{N=1,\;h_o \le h_\star,\;h_{\hat{u}}\le h'' \} \right].\label{eq:nh_o}
\end{align}
On the event~$\{N=1,\;h_o \le h_\star,\; h_{\hat{u}} \le h''\}$, conditioned on the respective locations~$v$ and~$u$ of the vertices~$o$ and~$\hat{u}$, the graph~$\mathbf{G}_\infty$ is stochastically smaller than the graph~$\mathbf{G}^{u,v}$ of Proposition~\ref{prop:guv}. Indeed, the conditioning gives no information on the graph apart from the locations of these two vertices~$o,\hat{u}$, and some negative information about the presence of other vertices  in the region~$\{(x,h)\in \mathbb{H}: |x-x_o| \le \exp\{(h_o + h)/2\}\}$. Hence, Proposition~\ref{prop:guv} gives
$$\mathbb{P}\left((\eta_t)_{t \ge \tau}\text{ survives}\mid o,\hat{u}\right) \le \lambda^{\varepsilon}\quad \text{on }\{N = 1,\;h_o \le h_\star,\; h_{\hat{u}} \le h''\}.$$
Then,~\eqref{eq:nh_o} is smaller than
$$\lambda^{1+\varepsilon} \cdot \mathbb{P}(N=1) \stackrel{\eqref{eq:bound_NN}}{\le} \lambda^{(2-\sigma)\left(2\alpha - 1\right)+1 + \varepsilon}.$$
If~$\sigma$ is small enough (depending on~$\varepsilon$), this is smaller than~$\lambda^{4\alpha - 1 + \varepsilon/2}$ for~$\lambda$ small enough, {and the proof of Proposition~\ref{upper341} is finished.}


\section{Convergence of density}\label{sec.convergence}

We prove here Theorem~\ref{thm:convergence}, that is the convergence in probability of the empirical density of infected sites to $\gamma(\lambda)$. We start with the upper bound. 
\begin{lemma}\label{lem.density.1}
Let~$(t_n)_{n \geq 1}$ be any sequence with~$t_n \to \infty$. Then, for any~$\varepsilon > 0$, and any $\lambda>0$, 
$$\lim_{n \to \infty}\P\left(\frac{|\xi_{t_n}^{{\bf G}_n}|}{|{\bf G}_n|} > \gamma(\lambda)+\varepsilon \right)  = 0.$$
\end{lemma}
\begin{proof}
Observe first that for any $R>0$, almost surely, 
$$\P(\xi_s^\rho \neq \emptyset, \, \xi_s^\rho \subseteq \B_\infty(\rho,R) \ \textrm{for all }s>0) = 0.$$ 
Using the fact that $({\bf G}_n)_{n\ge 1}$ uniformly rooted, converges locally to $({\bf G}_\infty,\rho)$ by Lemma~\ref{BSconvergence}, this yields for any sequence $(t_n)_{n\ge 0}$, with $t_n\to \infty$, and any fixed $R>0$, 
\begin{equation}\label{conv.density.1bis}
\lim_{n\to \infty} \E\left[ \frac 1{|{\bf G}_n|} \sum_{v\in {\bf G}_n} {\bf 1}\{\xi_{t_n}^v\neq \emptyset, \, \xi_s^v \subseteq \B_n(v,R) \ \textrm{for all }s\le t_n\}\right] = 0.
\end{equation}
We thus have by self-duality of the contact process (recall \eqref{eq:duality_formula}),  
\begin{align}\label{conv.density.2}
\nonumber & \P\left(\frac {|\xi_{t_n}^{{\bf G}_n}|}{|{\bf G}_n|} >\gamma(\lambda) + \varepsilon \right)  = \P\left(\frac 1{|{\bf G}_n|} \sum_{v\in {\bf G}_n}{\bf 1}\{\xi_{t_n}^v\neq \emptyset\} >\gamma(\lambda) + \varepsilon \right) \\ 
\nonumber & \stackrel{\eqref{conv.density.1bis}}{\le}  \P\left(\frac 1{|{\bf G}_n|} \sum_{v\in {\bf G}_n}{\bf 1}\{\xi_{t_n}^v\neq \emptyset, \exists s\le t_n\, :\, \xi_s^v \nsubseteq \B_n(v,R)\} >\gamma(\lambda) + \frac{3\varepsilon}{4} \right)  + o(1) \\  
&\le \P\left(X_n >\gamma(\lambda) + \frac{3\varepsilon}{4} \right)  + o(1), 
\end{align}
with 
$$X_n:=\frac 1{|{\bf G}_n|} \sum_{v\in {\bf G}_n}{\bf 1}\{\exists s>0 \, :\, \xi_s^v \nsubseteq \B_n(v,R)\}.$$
We will then apply Chebyshev's inequality in order to bound the probability on the right-hand side of \eqref{conv.density.2}. For this we need bounds on the expectation and variance of $X_n$.  
Concerning the expectation, observe that almost surely,   
$$\bigcap_{R>0} \{\exists s>0\, :\, \xi_s^\rho \nsubseteq \B_\infty(\rho,R)\} \ \subseteq \ \{\xi_s^\rho \neq \emptyset \ \forall s>0\}.$$  
Indeed, if the process does not escape to infinity in finite time, then this is true by definition, and if it does, then in particular infinitely many vertices get infected, 
which in turn almost surely maintain the process alive for an infinite amount of time (just because for any $t>0$, almost surely at least one of them survives for a time larger than $t$). 
Therefore for any $\varepsilon >0$, there exists $R>0$, such that 
\begin{equation}\label{conv.density.0}
\P(\exists s>0\, :\, \xi_s^\rho \nsubseteq \B_\infty(\rho,R))\le \gamma(\lambda)+\varepsilon/4.
\end{equation}
Fix now $\varepsilon>0$, and then $R>0$ as above. Using again that $({\bf G}_n)_{n\ge 1}$ uniformly rooted converges locally to $({\bf G}_\infty,\rho)$, we deduce that  
\begin{align}\label{conv.density.1}
\lim_{n\to \infty}  \E\left[ X_n\right] =  \P( \exists s>0\, :\,  \xi_s^\rho \nsubseteq \B_\infty(\rho,R)).
\end{align}
Then \eqref{conv.density.0} and \eqref{conv.density.1} show that for the above choice of $R$, for $n$ large enough, 
\begin{equation}\label{conv.density.2bis}
\E[X_n] \le \gamma(\lambda) + \frac{\varepsilon}{2}.
\end{equation}
We move now to the variance of $X_n$. We first notice that $|{\bf G}_n|\sim n$, in probability. Indeed by definition $|{\bf G}_n|$ is a Poisson random variable with parameter $\mu(\cR_n)$, where we recall $\cR_n=[-\frac{\pi}{2}n,\frac{\pi}{2}n] \times [0,2\log n]$, 
and from the definition of $\mu$, one can easily verify that $\mu(\cR_n)\sim n$. 
We next subdivide $\cR_n$ into a disjoint union of small cubes $(B_{i,j})_{i,j}$ of side length one. 
More precisely, for $i\in \Z$ and $j\in \N$, we set $B_{i,j}:= [i,i+1]\times [j,j+1]$. Then let 
$$Z_{i,j} :=\sum_{v\in {\bf G}_n \cap B_{i,j}}{\bf 1}\{\exists s>0 \, :\, \xi_s^v \nsubseteq \B_n(v,R)\},$$
and 
$$\widetilde X_n:= \frac 1n \sum_{i,j} Z_{i,j}.$$
Due to the above discussion it suffices to 
show that for some constant $C>0$, for any $\varepsilon>0$, 
$$\P(\widetilde X_n -\E[\widetilde X_n]  \ge \varepsilon/5)\le C\varepsilon.$$ 
Let now $h_\varepsilon>0$ sufficiently large, be such that 
$$\mu([-\frac \pi 2n,\frac \pi 2n]\times [h_\varepsilon,2\log n]) \le \varepsilon^2.$$
Noting that one can bound $Z_{i,j}$ by $|{\bf G}_n\cap B_{i,j}|$, whose mean is exactly $\mu(B_{i,j})$, we get using Markov's inequality 
$$\P\left(\sum_{(i,j)\, :\, j\ge h_\varepsilon} Z_{i,j} \ge \frac{\varepsilon}{10}\right) \le 10 \varepsilon.$$
Thus all we need to show in fact is that 
\begin{equation}\label{goal.density.1}
\P(X_n^\varepsilon - \E[X_n^\varepsilon]\ge \frac{\varepsilon}{10}) = o(1), \quad \text{with}\quad X_n^\varepsilon := \frac 1n \sum_{(i,j)\, :\, j\le h_\varepsilon} Z_{i,j}.
\end{equation}
To this end, we estimate the variance of $X_n^\varepsilon$. Note that for any pairs of indices $(i,j)$ and $(k,\ell)$, conditionally on ${\bf G}_n$, $Z_{i,j}$ and $Z_{k,\ell}$ are independent, unless $B_{k,\ell}$ intersects the ball of radius $2R$ centered at some vertex of $B_{i,j}$. Moreover, in the latter case, one can use again the trivial bound 
$$|\text{Cov}(Z_{i,j},Z_{k,\ell})| \le  |{\bf G}_n \cap B_{i,j}| \cdot |{\bf G}_n \cap B_{k,\ell}|,$$
yielding 
\begin{align*}
\var(X_n^\varepsilon) & \le \frac 1{n^2} \sum_{(i,j)\, :\, j\le h_\varepsilon} \E\left[|{\bf G}_n \cap B_{i,j}|\cdot |{\bf G}_n \cap (\cup_{v\in {\bf G}_n\cap B_{i,j}} \B_n(v,2R+1)| \right]\\
& \le  \frac C{n^2} \sum_{(i,j)\, :\, j\le h_\varepsilon}  \mu(B_{i,j}) = \cO(\frac 1n),
\end{align*}
where $C=C(R,\varepsilon)=1+2\mu(\B_n((0,h_{\varepsilon}+1),2R+1))$, is a constant that only depends on $R$ and $\varepsilon$.  
Then \eqref{goal.density.1} follows and this concludes the proof of the lemma. 
\end{proof}

We prove now the lower bound, which is a bit more delicate. 

\begin{proposition}
Let~$(t_n)_{n \ge 1}$ be any sequence with~$t_n \to \infty$ and~$t_n < e^{cn}$ for each~$n$, with $c$ as in Theorem \ref{theo.exp}.  Then, for any~$\varepsilon > 0$ and $\lambda>0$,  
$$\lim_{n \to \infty}\P\left(\frac{|\xi^{\bf G_n}_{t_n}|}{|{\bf G}_n|} < \gamma(\lambda)-\varepsilon \right)  = 0.$$
\end{proposition}
\begin{proof}
Fix $\varepsilon>0$. Using that for any $R>0$, one has 
$$\P(\xi_s^\rho \subseteq \B_\infty(\rho,R), \ \text{for all }s>0) = 0,$$
we deduce as in the proof of the previous lemma, that for any sequence $(t_n)_{n\ge 1}$, with $t_n\to \infty$, 
$$
\P\left(\frac {|\xi_{t_n}^{{\bf G}_n}|}{|{\bf G}_n|} <\gamma(\lambda) - \varepsilon \right)  \le \P\left(\frac 1{|{\bf G}_n|} \sum_{v\in {\bf G}_n}{\bf 1}\{\xi_{t_n}^v\neq \emptyset, \exists s>0\, :\, \xi_s^v \nsubseteq \B_n(v,R)\} <\gamma(\lambda)  - \frac{3\varepsilon}{4} \right)  + o(1).$$ 
Moreover, as before,  for any $\varepsilon > 0$, there exists $R>0$, such that
$$\P\left(\exists s>0\, :\, \xi_s^\rho \nsubseteq \B_\infty(\rho,R)\right) \ge \gamma(\lambda)-\frac{\varepsilon}{4}.$$
Then, using the same argument as in the proof of the previous lemma, we get
$$\lim_{n\to \infty} \P\left(\frac 1{|{\bf G}_n|} \sum_{v\in {\bf G}_n}{\bf 1}\{\exists s>0\, :\, \xi_s^v \nsubseteq \B_n(v,R)\} \le \gamma(\lambda)  - \frac{\varepsilon}{2} \right)   = 0.$$
Thus, 
\begin{equation}\label{lower.density}
 \P\left(\frac {|\xi_{t_n}^{\bf G_n}|}{|{\bf G}_n|} <\gamma(\lambda) - \varepsilon \right)\le \P\left(\frac 1{|{\bf G}_n|} \sum_{v\in \bf G_n}{\bf 1}\{\xi_{t_n}^v= \emptyset, \exists s>0\, :\, \xi_s^v \nsubseteq \B_n(v,R)\} > \frac{\varepsilon}{4} \right)  + o(1).
 \end{equation} 
We proceed now as in the previous lemma, but this time we only need a first moment bound. Recall the notation for $B_{i,j}$, from there, and let  
$$Y_n:=\frac 1n \sum_{i,j} {\bf 1}(E_{i,j}) \cdot |{\bf G}_n \cap B_{i,j}|,$$
where 
$$E_{i,j}:= \{\exists v\in {\bf G}_n \cap B_{i,j}\, :\, \xi_{t_n}^v= \emptyset\text{ and } \exists s>0 \text{ with } \xi_s^v \nsubseteq \B_n(v,R)\}.$$
We claim that when $R$ is large enough, one has almost surely,  
\begin{equation}\label{claim.density} 
\P\left(E_{i,j} \mid {\bf G}_n \cap B_{i,j}\right) \le \varepsilon^2. 
\end{equation} 

Note that given this fact we deduce that for some constant $C>0$, 
$$\E[Y_n] \le C\varepsilon^2,$$
and together with Markov's inequality, we get that the first term on the right-hand side of \eqref{lower.density} is $\cO(\varepsilon)$, from which the proposition follows.

Let us prove now~\eqref{claim.density}. The basic idea is quite simple: each time the process reaches a new shell $\B_n(v,i+1)\setminus \B_n(v,i)$, 
it has some positive probability to infect a vertex at some high level, which will then 
sustain the infection for a time $t_n$ with high probability, as was shown in the proof of Theorem \ref{theo.exp}. 
If $R$ is taken large enough, then the process will have many chances to do this, and thus it should happen with probability as close to one as wanted.

We proceed now with the details which require a certain care due to the hyperbolic shape of the balls.
Define $h_\ell$, for each $\ell\ge 1$, by 
$$\mu([0,2^\ell]\times [h_\ell,\infty)) = \frac{\varepsilon^3}{2\ell^2},$$
or equivalently by 
$$h_\ell = \frac 1\alpha ((\ell+1) \log 2 - 3\log \varepsilon + 2\log \ell).$$   
Note that by Markov's inequality, for each $\ell \ge 1$, 
$$\P({\bf G}_n\cap [0,2^\ell]\times [h_\ell,2\log n]\neq \emptyset) \le \E \left[ |{\bf G}_n\cap [0,2^\ell]\times [h_\ell,2\log n]|\right]\le \frac{\varepsilon^3}{2\ell^2}.$$
Thus, letting  
$$x_\ell := x_0+\sum_{m=1}^\ell 2^m,\quad  \text{and} \quad D_\ell : = [x_\ell,x_{\ell+1}]\times [h_{\ell+1},2\log n],$$ 
a union bound gives
\begin{equation}\label{Anv}
\P\left(\cA_n(v)\right) \ge 1- \varepsilon^3,\quad \text{where}\quad \cA_n(v):=\{{\bf G}_n\cap (\cup_{\ell \ge 0} D_\ell)= \emptyset\}.
\end{equation}
Let $L:=\frac{\alpha +1}{2\alpha}\cdot \log 2$, be as in Section~\ref{Section.exp}, and then for $\ell \ge 0$, set 
$$Q(\ell):= [x_\ell,x_{\ell +1}]\times [(\ell+2) L,(\ell +3)L].$$ 
Note the important property of these boxes, which is that any vertex in ${\bf G}_n\cap Q(\ell)$ is a neighbor of any other vertex in 
${\bf G}_n\cap Q(\ell+1)$, for any $\ell\ge 0$ (this follows from the fact that $L>\log 2$). 
We call $\cC_n(v)$ the event when all these boxes are good in the sense of Section~\ref{Section.exp}, at least for $\ell$ large enough. That is, we define 
$$\cC_n(v):= \{|{\bf G}_n\cap Q(\ell) | \ge C_0\lambda^{-3}  \quad \text{for all }\ell \ge \ell_0\},$$
where $C_0$ is a positive constant to be fixed later, and $\ell_0$ is the smallest integer such that 
\begin{equation}\label{Bnv}
\P(\cC_n(v))\ge 1- \varepsilon^3.
\end{equation}
Observe that for any $0\le \ell\le m$, one has $(x_m,h_m)$ and $(x_\ell,h_\ell)$ are neighbors in ${\bf G}_n$, only if $m\le \frac{\ell}{2\alpha -1} + C\log m$, 
for some constant $C>0$. In particular, since $\alpha>1/2$, for any fixed $\ell$, this happens only for finitely many integers $m\ge \ell$, and one can thus define inductively the sequence $(\ell_i)_{i\ge 0}$, by  $\ell_0= 0$, and for $i\ge 0$, 
$$\ell_{i+1} = \inf\{m> \ell_i \, :\, (x_{m'},h_{m'})\notin \B_n((x_{\ell_i+1},h_{\ell_i+1}),1) \quad \forall m'\ge m\}.$$

Consider now $(\xi_t^v)_{t\ge 0}$ the contact process starting from only $v$ infected, and define: 
$$\cH_n(v):= \{\exists s\ge 0\, :\, \xi_s^v \cap (\cup_{\ell\ge \ell_0} Q(\ell)) \neq \emptyset\}.$$
The proof of Theorem \ref{theo.exp} given in Section \ref{Section.exp} shows that (at least by taking $C_0$ large enough)
$$\P(\cC_n(v) \cap \H_n(v)\cap \{\xi_{t_n}^v=\emptyset\}) \le \varepsilon^3.$$
Therefore, recalling \eqref{Anv} and \eqref{Bnv}, we see that all we need to show is that for $R$ large enough,  
\begin{equation}\label{goal}
\P(\H_n(v)^c \cap  \cA_n(v)\cap \{\exists s>0\, :\, \xi_s^v \nsubseteq \B_n^+(v,R)\})\le \varepsilon^3,
\end{equation}
where $\B_n^+(v,R) :=\B_n(v,R)\cup \left([-\frac{\pi}{2}n,x_0]\times [0,2\log n]\cap {\bf G}_n\right)$.
Indeed, this would show that, for $\varepsilon$ small enough, 
$$\P(\{\exists s>0\, :\, \xi_s^v \nsubseteq \B_n^+(v,R)\}\cap\{\xi_{t_n}^v=\emptyset\} )\le 3 \varepsilon^3\le \varepsilon^2,$$
and as explained previously this would conclude the proof of the proposition.

We prove now \eqref{goal}. For $i\ge 0$, we define the stopping time 
$$\tau_i := \inf \{s>0:\ \exists w=(x,h)\in \xi_s^v, \text{ with }x\ge x_{\ell_{2i}} \}.$$
Note that when the rectangles $D_\ell$ are empty, then the first coordinate of the vertex which is infected at time $\tau_i$ cannot be larger than  $x_{\ell_{2i+2}}$. Otherwise there would exist $m\ge \ell_{2i+2}$, such that $(x_m,h_m)$ would be in the neighborhood of 
$(x_{\ell_{2i}},h_{\ell_{2i}})$, and this is not possible by definition of the sequence $(\ell_j)_{j\ge 0}$. 
In other words, for any $i\ge 0$, on the event $\cA_n(v)\cap \{\tau_i<\infty\}$, one has $\tau_i<\tau_{i+1}$.

We then consider the {\it good} events  
$$A_i^1:=\{|{\bf G}_n\cap [x_{\ell_{2i}},x_{\ell_{2i}}+2^j]\times [jL,(j+1)L]| \ge \lambda^{-3} \quad \text{for all }\ell_0\le j\le \ell_{2i}+1\}, $$ 
and 
$$A_i^2:= \{ {\bf G}_n\cap [x_{\ell_{2i}},x_{\ell_{2i}}+2^j]\times [jL,(j+1)L] \neq \emptyset \quad \text{for all }0\le j\le \ell_0\},$$ 
and set 
$$A_i:=A_i^1\cap A_i^2.$$
We next define $B_i$, as the event that $\tau_i$ is finite and that after this time, there exists an infection path within the rectangle $[x_{\ell_{2i-1}+1},x_{\ell_{2i+1}}]\times [0,2\log n]$, going from the 
vertex infected at time $\tau_i$ up to a vertex in $Q(\ell_{2i})$. We also need to consider truncated versions of $\cA_n(v)$, defined for any $i$, by
$$\cA_n^i(v):=\{{\bf G}_n\cap (\cup_{\ell \le \ell_{2i+2}} D_\ell)= \emptyset\}.$$
We finally consider the filtration $(\cG_i)_{i\ge 0}$, where $\cG_i$ is the $\sigma$-field generated by this set $\cA_n^i(v)$, the restriction of the graph ${\bf G}_n$ to the rectangle 
$[x_0,x_{\ell_{2i+1}}]\times [0,2\log n]$, 
together with all the Poisson clocks associated to the vertices in this rectangle, as well as all those associated to the edges between them in the Harris construction. Note that by definition $A_i$ is $\cG_i$-measurable. 
Note also that by definition of the $(\ell_j)_{j\ge 0}$, the event $\cA_n^{i-1}(v)\cap B_i$ is $\cG_i$-measurable as well, since on $\cA_n^{i-1}(v)$, the vertex infected at time 
$\tau_i$ has a first coordinate smaller than $x_{\ell_{2i+1}}$.

Moreover, by definition 
$$B_i\subseteq \cH_n(v),\quad \text{for all }i\ge 1,$$
and therefore for any integer $r\ge 1$, 
\begin{equation}\label{HnvBi}
\cH_n(v)^c \subseteq\ \bigcap_{i\le r} B_i^c.
\end{equation} 
On the other hand, a straightforward computation shows that there exists a constant $p_1>0$, independent of $i$, such that almost surely, 

$$\P(A_i\mid \cG_{i-1})=\P(A_i) \ge p_1,$$
using for the first equality that $A_i$ is independent of $\cG_{i-1}$, by definition.

Now we claim that on the event $A_i\cap \cA_n^{i-1}(v)\cap \{\tau_i<\infty\}$, the vertex infected at time $\tau_i$, or the one who infected it, has a neighbor (possibly itself) in 
one of the boxes occurring in the definition of $A_i^1$ or $A_i^2$. Indeed, let $v_i=(x_i,h_i)$ be the vertex infected at time $\tau_i$ and $v'_i=(x'_i,h'_i)$ be the one who infected it. By definition one has $x_i\ge x_{\ell_{2i}}$, and $x'_i< x_{\ell_{2i}}$. Since $v_i$ and $v'_i$ are neighbors, one also has $|x_i-x'_i|\le e^{(h_i + h'_i)/2}$. Assume first that $h_i\ge h'_i$, and let $j\ge 0$ be such that $jL \le h_i<(j+1)L$. 
Note that one can assume $x_i > x_{\ell_{2i}} + 2^j$, as otherwise there is nothing to prove (since in this case $v_i$ already belongs to one of the boxes appearing in the definition of $A_i^1$ and $A_i^2$). 
Now by definition on the event $A_i$ there exists $v''_i=(x''_i,h''_i)\in {\bf G}_n$, such that 
$x_{\ell_{2i}} \le x''_i\le x_{\ell_{2i}} + 2^{j+1}$, and $(j+1)L\le h''_i\le (j+2)L$. 
Note that one has either $x'_i <x''_i\le x_i$, or $0\le x''_i - x_i \le 2^j < x_i - x'_i$.  
Hence, in all cases it holds  
$$|x''_i - x_i| \le |x_i - x'_i| \le e^{(h_i + h'_i)/2}\le e^{(h_i+h''_i)/2},$$
and thus $v_i$ and $v''_i$ are neighbors, which proves our claim when $h_i\ge h'_i$. 
If on the other hand $h_i\le h'_i$, then we can use a similar argument: 
assume $jL\le h_i<(j+1)L$, for some $j\ge 0$, and again that $x_i>x_{\ell_{2i}}+2^j$, as otherwise there is nothing to prove. 
Pick a vertex $v''_i=(x''_i,h''_i)$ in ${\bf G}_n\cap [x_{\ell_{2i}},x_{\ell_{2i}}+ 2^{j+1}] \times [(j+1)L,(j+2)L]$.  
If $x_i>x''_i$, then 
$$|x'_i - x''_i|\le |x'_i - x_i| \le e^{(h_i+h'_i)/2}\le e^{(h'_i + h''_i)/2},$$ 
which implies that $v'_i$ and $v''_i$ are neighbors. If 
$x_i<x''_i$, then 
$$|x_i-x''_i|\le 2^j \le e^{jL} \le e^{(h_i + h''_i)/2},$$
using for the second inequality that $L>\log 2$, since $\frac{1+\alpha}{2\alpha}>1$, for any $\alpha<1$. This proves the 
claim in the case $h_i\ge h'_i$ as well.

It follows that after time $\tau_i$, the vertex $v_i$ will infect another vertex in one of the cubes occurring in the definition of $A_i^1$ or $A_i^2$, with probability at least $(\lambda/(1+\lambda))^2$. 
Once infected it will propagate the infection up to $Q(\ell_i)$ within the boxes appearing in the definition of $A_i^1$ and $A_i^2$ with positive probability, uniformly bounded from below by a constant independent of $i$ (this last point following from the same argument as in the proof of Theorem~\ref{theo.exp}). 
Therefore, there also exists a constant $p_2\in (0,1)$, such that on $\cA_n^{i-1}(v)$, 
$$\P(A_i\cap B_i^c \cap \{\tau_i<\infty\} \mid \cG_{i-1}) \le p_1 p_2.$$
As a consequence, there exists $p\in (0,1)$, such that on $\cA_n^{i-1}(v)$, one has 
\begin{equation}\label{induction.Bi}
\P(B_i^c\cap \{\tau_i<\infty\} \mid \cG_{i-1})\le 1-p, \quad \text{for all }i\ge 1.
\end{equation}
The conclusion follows: indeed, let first $r$ be some integer such that $(1-p)^r\le \varepsilon^3/2$, and note that for $R$ large enough, 
$$\{\exists s>0\, :\, \xi_s^v \nsubseteq \B_n(v,R)\} \subset \{\tau_r<\infty \} \cup \{\tau_{-r}<\infty\},$$
where we denote by $\tau_{-r}$ the first time when there is an infected vertex with $x$-coordinate smaller than $-x_{\ell_{2r}}$. By symmetry we can consider only the event $\{\tau_r<\infty\}$, but then \eqref{induction.Bi} and an immediate induction give 
$$\P\left(\cA_n(v), \, \tau_r<\infty,\, \cap_{i\le r} B_i^c\right)\le \varepsilon^3/2,$$
from which \eqref{goal} follows using also \eqref{HnvBi}. This concludes the proof of the proposition. 
\end{proof}

\section{Discussion and outlook}
In this paper we gave a complete picture of metastability for $\frac12 < \alpha < 1$. Naturally, one might wonder how the contact process evolves outside this regime: on the one hand, for $\alpha < \frac12$, the total number of edges of $G_n$ is superlinear, and hence, we do not expect metastability in this case (and there is no natural infinite graph either); a similar phenomenon could also arise in the case $\alpha=\frac12$. On the other hand, for $\alpha > 1$, the largest component is of order $n^{1/(2\alpha)} \ll \sqrt{n}$, roughly corresponding to the maximum degree (see~\cite{Diel}). Therefore, this component is roughly like a star, with a few extra edges. The same proof given therein can be used to show that most other components are star-like, and there should be of the order $n^{1-2\alpha \beta}$ such star-like components of size $n^{\beta}$ for any $0 < \beta \le 1/(2\alpha)$. Hence, the expected component size in the infinite graph is of order $\int_{\beta=0}^{1/(2\alpha)} n^{2-2\alpha} d\beta$, which is finite for $\alpha > 1$. Thus, the component of the root is almost surely finite, and the contact process cannnot survive. For $\alpha=1$, for $\nu$ sufficiently large (see~\cite{FM}) there exists a giant component, and the study of the contact process in this regime is subject to further work.

\end{document}